\documentclass[12pt,reqno]{amsart}

\usepackage{color}
\usepackage{amsmath, amsthm, amssymb}
\usepackage{amsfonts}
\usepackage[ansinew]{inputenc}
\usepackage[dvips]{epsfig}
\usepackage{graphicx}
\usepackage[english]{babel}
\usepackage{hyperref}

\theoremstyle{plain}
\newtheorem{thm}{Theorem}[section]
\newtheorem{cor}[thm]{Corollary}
\newtheorem{lem}[thm]{Lemma}

\theoremstyle{definition}

\theoremstyle{remark}
\newtheorem{rem}[thm]{Remark}

\numberwithin{equation}{section}

\newcommand{\average}{{\mathchoice {\kern1ex\vcenter{\hrule height.4pt
width 6pt depth0pt} \kern-9.7pt} {\kern1ex\vcenter{\hrule
height.4pt width 4.3pt depth0pt} \kern-7pt} {} {} }}

\def\iden{\text{I}}
\def\R{\mathbb{R}}

\def\div{\text{div}}

\begin{document}

\title{Sharp isoperimetric inequalities via the ABP method}

\author{Xavier Cabr\'e}

\address{ICREA and Universitat Polit\`ecnica de Catalunya, Departament de Matem\`{a}tica  Aplicada I, Diagonal 647, 08028 Barcelona, Spain}
\email{xavier.cabre@upc.edu}

\author{Xavier Ros-Oton}

\address{Universitat Polit\`ecnica de Catalunya, Departament de Matem\`{a}tica  Aplicada I, Diagonal 647, 08028 Barcelona, Spain}
\email{xavier.ros.oton@upc.edu}

\author{Joaquim Serra}

\address{Universitat Polit\`ecnica de Catalunya, Departament de Matem\`{a}tica  Aplicada I, Diagonal 647, 08028 Barcelona, Spain}

\email{joaquim.serra@upc.edu}

\thanks{The authors were supported by grants MINECO MTM2011-27739-C04-01 and GENCAT 2009SGR-345}

\keywords{Isoperimetric inequalities, densities, convex cones, homogeneous weights, Wulff shapes, ABP method.}

\subjclass[2010]{Primary 28A75; Secondary 35A23, 49Q20}

\begin{abstract}
We prove some old and new isoperimetric inequalities with the best constant using the ABP method
applied to an appropriate linear Neumann problem.
More precisely, we obtain a new family of sharp isoperimetric inequalities
with weights (also called densities) in open convex cones of $\R^n$.
Our result applies to all nonnegative homogeneous weights satisfying a concavity
condition in the cone. Remarkably, Euclidean balls centered at the origin
(intersected with the cone) minimize the weighted isoperimetric quotient,
even if all our weights are nonradial ---except for the constant ones.

We also study the anisotropic isoperimetric problem in convex cones for the same class of
weights. We prove that the Wulff shape (intersected with the cone)
minimizes the anisotropic weighted perimeter under the weighted volume constraint.

As a particular case of our results, we give new proofs of two classical results:
the Wulff inequality and the isoperimetric inequality in convex cones of Lions and Pacella.
\end{abstract}

\maketitle

\section{Introduction and results}


In this paper we study isoperimetric problems with weights ---also called densities.
Given a weight $w$ (that is, a positive function $w$), one wants to characterize
minimizers of the weighted perimeter $\int_{\partial E} w$ among those sets $E$
having weighted volume $\int_{E} w$ equal to a given constant.
A set solving the problem, if it exists, is called an isoperimetric set or
simply a minimizer. This question, and the associated isoperimetric inequalities with weights,
have attracted much attention recently;
see for example \cite{MP}, \cite{MM}, \cite{CP}, \cite{FM}, and \cite{M}.

The solution to the isoperimetric problem in $\R^n$ with a
weight $w$ is known only for very few weights, even in the case $n=2$.
For example, in
$\R^n$ with the Gaussian weight $w(x)=e^{-|x|^2}$ all the minimizers are half-spaces \cite{B,CFMP},
and with $w(x)=e^{|x|^2}$ all the minimizers are balls centered at the origin \cite{RCBM}.
Instead, mixed Euclidean-Gaussian densities lead to minimizers that have a more intricate
structure of revolution \cite{FMP}.
The radial homogeneous weight $|x|^\alpha$ has been considered very recently.
In the plane ($n=2$), minimizers for this homogeneous weight depend on the values of $\alpha$.
On the one hand, Carroll-Jacob-Quinn-Walters \cite{CJQW} showed that when $\alpha<-2$ all minimizers are $\R^2\setminus B_r(0)$, $r>0$,
and that when $-2\leq \alpha<0$ minimizers do not exist.
On the other hand, when $\alpha>0$ Dahlberg-Dubbs-Newkirk-Tran \cite{DDNT} proved that all minimizers are circles passing through the origin
(in particular, not centered at the origin).
Note that this result shows that even radial homogeneous weights may lead to nonradial minimizers.

Weighted isoperimetric inequalities in cones have also been considered.
In these results, the perimeter of $E$ is taken relative to the cone, that is, not counting the part of $\partial E$ that lies on the boundary of the cone.
In \cite{DHHT} D\'iaz-Harman-Howe-Thompson consider again the radial homogeneous weight $w(x)=|x|^\alpha$, with $\alpha>0$, but now in an open convex cone $\Sigma$ of angle $\beta$ in the plane $\R^2$.
Among other things, they prove that there exists $\beta_0\in (0,\pi)$ such that for $\beta<\beta_0$ all minimizers are $B_r(0)\cap \Sigma$, $r>0$, while these circular sets about the origin are not minimizers for $\beta>\beta_0$.

Also related to the weighted isoperimetric problem in cones, the following is a recent result by Brock-Chiaccio-Mercaldo \cite{BCM}.
Assume that $\Sigma$ is any cone in $\R^n$ with vertex at the origin, and consider the isoperimetric problem in $\Sigma$ with any weight $w$.
Then, for $B_R(0)\cap\Sigma$ to be an isoperimetric set for every $R>0$ a necessary condition is that $w$ admits the factorization
\begin{equation}\label{weightproduct}w(x)=A(r)B(\Theta),\end{equation}
where $r=|x|$ and $\Theta=x/r$.
Our main result ---Theorem \ref{th1} below--- gives a sufficient condition on $B(\Theta)$ whenever $\Sigma$ is convex and $A(r)=r^\alpha$, $\alpha\geq0$, to guarantee that $B_R(0)\cap \Sigma$ are isoperimetric sets.

Our result states that Euclidean balls centered at the origin solve the isoperimetric problem in any open convex cone $\Sigma$
of $\R^n$ (with vertex at the origin) for a certain class of nonradial weights.
More precisely, our result applies to all nonnegative continuous
weights $w$ which are positively homogeneous of degree $\alpha\geq0$ and such that $w^{1/\alpha}$ is concave in the cone $\Sigma$ in case $\alpha>0$.
That is, using the previous notation, $w=r^\alpha B(\Theta)$ must be continuous in $\overline\Sigma$ and $rB^{1/\alpha}(\Theta)$ must be concave in $\Sigma$.
We also solve weighted \emph{anisotropic} isoperimetric problems in cones for the same class of weights.
In these weighted anisotropic problems, the perimeter of a domain $\Omega$ is given by
\[\int_{\partial\Omega\cap\Sigma} H(\nu(x))w(x)dS,\]
where $\nu(x)$ is the unit outward normal to $\partial\Omega$ at $x$, and $H$ is a positive,
positively homogeneous of degree one, and convex function.
Our results were announced in the recent note \cite{CRS-CRAS}.

In the isotropic case, making the first variation of weighted perimeter (see \cite{RCBM}), one sees that the (generalized) mean curvature of $\partial\Omega$ with the density $w$ is
\begin{equation}\label{genmean}
H_w=H_{\rm eucl}+\frac{1}{n}\frac{\partial_\nu w}{w},
\end{equation}
where $\nu$ is is the unit outward normal to $\partial\Omega$ and $H_{\rm eucl}$ is the Euclidean mean
curvature of $\partial\Omega$.
It follows that balls centered at the origin intersected with the cone have constant mean curvature whenever
the weight is of the form \eqref{weightproduct}.
However, as we have seen in several examples presented above, it is far from being true that the solution
of the isoperimetric problem for all the weights satisfying \eqref{weightproduct} are balls centered at
the origin intersected with the cone.
Our result provides a large class of nonradial weights for which, remarkably,
Euclidean balls centered at the origin (intersected with the cone) solve the isoperimetric problem.

In Section \ref{examples} we give a list of weights $w$ for which our result applies.
Some concrete examples are the following:
\[\textrm{dist}(x,\partial\Sigma)^\alpha\quad \textrm{  in } \Sigma\subset \R^n,\]
where $\Sigma$ is any open convex cone and $\alpha\geq0$ (see example (ii) in Section \ref{examples});
\[x^{a}y^b z^{c},\ (ax^r+by^r+cz^r)^{\alpha/r},\ \textrm{ or }\ \frac{xyz}{xy+yz+zx}\quad
\textrm{ in } \Sigma=(0,\infty)^3,\]
where $a,b,c$ are nonnegative numbers, $r\in(0,1]$ or $r<0$, and $\alpha>0$ (see examples (iv), (v), and (vii), respectively);
\[\frac{x-y}{\log x-\log y},\ \frac{x^{a+1}y^{b+1}}{\left(x^p+y^p\right)^{1/p}},
\ \textrm{ or }\ x\log\left(\frac{y}{x}\right)\quad \textrm{ in } \Sigma=(0,\infty)^2,\]
where $a\geq0$, $b\geq0$, and $p>-1$ (see examples (viii) and (ix));
\[\left(\frac{\sigma_l}{\sigma_k}\right)^{\frac{\alpha}{l-k}},\quad 1\leq k<l<n,\quad \textrm{ in } \Sigma=\{\sigma_1>0,\ ...,\sigma_l>0\},\]
where $\sigma_k$ is the elementary symmetric function of order $k$, defined by $\sigma_k(x)=\sum_{1\leq i_1<\cdots<i_k\leq n}x_{i_1}\cdots x_{i_k}$, and $\alpha>0$ (see example (vii)).

Our isoperimetric inequality with an homogeneous weight $w$ of degree $\alpha$
in a convex cone $\Sigma\subset\R^n$ yields as a
consequence the following Sobolev inequality with best constant.
If $D=n+\alpha$, $1\leq p<D$, and $p_*=\frac{pD}{D-p}$, then
\begin{equation}\label{Sob}
\left(\int_{\Sigma}|u|^{p_*}w(x)dx\right)^{1/p_*}\leq C_{w,p,n}
\left(\int_{\Sigma}|\nabla u|^pw(x)dx\right)^{1/p}
\end{equation}
for all smooth functions $u$ with compact support in $\R^n$ ---in particular,
not necessarily vanishing on $\partial\Sigma$.
This is a consequence of our isoperimetric
inequality, Theorem \ref{th1}, and a weighted radial rearrangement
of Talenti \cite{T}, since these two results yield the radial symmetry of minimizers.

The proof of our main result follows the ideas introduced by the first author \cite{CSCM,CDCDS} in a new proof of the classical isoperimetric inequality (the classical isoperimetric
inequality corresponds to the weight $w\equiv1$ and the cone $\Sigma=\R^n$).
Our proof consists of applying the ABP method
to an appropriate linear Neumann problem involving the operator
\[w^{-1}\textrm{div}(w\nabla u)=\Delta u+\frac{\nabla w}{w}\cdot \nabla u,\]
where $w$ is the weight.

\medskip
\subsection{\sc The setting}\label{subs1.1}
\hspace{\fill}
\smallskip

The classical isoperimetric problem in convex cones was solved by P.-L. Lions and F. Pacella \cite{LP} in 1990.
Their result states that among all sets $E$ with fixed volume inside an open convex cone $\Sigma$,
the balls centered at the vertex of the cone minimize the perimeter relative to the cone
(recall that the part of the boundary of $E$ that lies on the boundary of the cone is not counted).

Throughout the paper $\Sigma$ is an open convex cone in $\R^n$. Recall that given a measurable set $E\subset \R^n$ the relative perimeter of $E$ in $\Sigma$ is defined by
\[P(E;\Sigma) := \sup \left\{\int_E \div\,\sigma \,dx\,:\,\sigma\in C^1_c(\Sigma, \R^n),\ |\sigma|\le 1 \right\}.\]
When $E$ is smooth enough,
\[P(E;\Sigma)=\int_{\partial E\cap\Sigma}dS.\]
The isoperimetric inequality in cones of Lions and Pacella reads as follows.

\begin{thm}[\cite{LP}] \label{isopcon}
Let $\Sigma$ be an open convex cone in $\mathbb R^n$ with vertex at $0$, and $B_1:=B_1(0)$.
Then,
\begin{equation}\label{pp}
\frac{P(E;\Sigma)}{|E\cap\Sigma|^{\frac{n-1}{n}} }\ge \frac{P(B_1;\Sigma)}{|B_1\cap\Sigma|^{\frac{n-1}{n}}}
\end{equation}
for every measurable set $E\subset \R^n$ with $|E\cap\Sigma|<\infty$.
\end{thm}

The assumption of convexity of the cone can not be removed. In the same
paper \cite{LP} the authors give simple examples of nonconvex cones for which
inequality \eqref{pp} does not hold.
The idea is that for cones having two disconnected
components, \eqref{pp} is false since it pays less perimeter to concentrate all the volume in
one of the two subcones.
A connected (but nonconvex) counterexample is then obtained by joining the two components by a conic open thin set.

The proof of Theorem \ref{isopcon} given in \cite{LP} is based on the Brunn-Minkowski inequality
\[|A+B|^{\frac1n}\geq |A|^{\frac1n}+|B|^{\frac1n},\]
valid for all nonempty measurable sets $A$ and $B$ of $\mathbb R^n$ for which $A+B$ is also measurable; see \cite{G} for more information on this inequality.
As a particular case of our main result, in this paper we provide
a totally different proof of Theorem~\ref{isopcon}.
This new proof is based on the ABP method.

Theorem \ref{isopcon} can be deduced from a degenerate case of the classical Wulff inequality stated in Theorem \ref{wulffthm} below.
This is because the convex set $B_1\cap\Sigma$ is the Wulff shape \eqref{wulffshape} associated to some appropriate anisotropic perimeter.
As explained below in Section \ref{elementsproof}, this idea is crucial in the proof of our main result.
This fact has also been used recently by Figalli and Indrei \cite{FI} to prove a quantitative isoperimetric inequality in convex cones.
From it, one deduces that balls centered at the origin are the unique minimizers in \eqref{pp} up to translations that leave invariant the cone (if they exist).
This had been established in \cite{LP} in the particular case when $\partial\Sigma\setminus\{0\}$ 
is smooth (and later in \cite{RiRo}, which also classified stable hypersurfaces in smooth cones).

The following is the notion of anisotropic perimeter.
We say that a function $H$ defined in $\R^n$ is a \emph{gauge} when
\begin{equation}\label{gauge}
H\textrm{ is nonnegative, positively homogeneous of degree one, and convex}.
\end{equation}
Somewhere in the paper we may require a function to be homogeneous; by this we always mean positively
homogeneous.

Any norm is a gauge, but a gauge may vanish on some unit vectors.
We need to allow this case since it will occur in our new proof of Theorem \ref{isopcon}
---which builds from the cone $\Sigma$ a gauge that is not a norm.

The anisotropic perimeter associated to the gauge $H$ is given by
\[P_{H}(E) := \sup \left\{\int_E \div\,\sigma \,dx\,:\, \sigma\in C^1_c(\R^n, \R^n),\ \sup_{H(y)\le 1} \left(\sigma(x)\cdot y\right)\le 1\ \textrm{for}\ x\in\R^n  \right\},\]
where $E\subset\R^n$ is any measurable set.
When $E$ is smooth enough one has
\[P_{H}(E) = \int_{\partial E}H\bigl(\nu(x)\bigr)dS,\]
where $\nu(x)$ is the unit outward normal at $x\in\partial E$.

The Wulff shape associated to $H$ is defined as
\begin{equation}\label{wulffshape}
W=\{ x\in\mathbb R^n \ :\ x\cdot\nu < H(\nu)\ \text{ for all }\nu\in S^{n-1}\}.
\end{equation}
We will always assume that $W\neq \varnothing$. Note that $W$ is an open set with $0\in\overline W$.
To visualize $W$, it is useful to note that it is the intersection of the half-spaces
$\{x\cdot\nu<H(\nu)\}$ among all $\nu\in S^{n-1}$.
In particular, $W$ is a convex set.

From the definition \eqref{wulffshape} of the Wulff shape it follows that,
given an open convex set $W\subset\R^n$ with $0\in\overline W$,
there is a unique gauge $H$ such that $W$ is the Wulff shape associated to $H$.
Indeed, it is uniquely defined by
\begin{equation}\label{uniqueH}
H(x)=\inf\bigl\{t\in\R:\ W\subset\{z\in\R^n\,:\,z\cdot x<t\}\bigr\}.
\end{equation}

Note that, for each direction $\nu\in S^{n-1}$, $\{x\cdot\nu=H(\nu)\}$ is a supporting hyperplane of $W$.
Thus, for almost every point $x$ on $\partial W$ ---those for which the outer normal $\nu(x)$ exists--- it holds
\begin{equation}\label{normalW}
x\cdot\nu(x)=H(\nu(x))\quad \textrm{a.e. on}\ \partial W.
\end{equation}
Note also that, since $W$ is convex, it is a Lipschitz domain.
Hence, we can use the divergence theorem to find the formula
\begin{equation}\label{per/vol}
P_{H}(W)=\int_{\partial W} H(\nu(x))dS=\int_{\partial W}x\cdot\nu(x)dS=\int_W \div(x)dx = n|W|,
\end{equation}
relating the volume and the anisotropic perimeter of $W$.

When $H$ is positive on $S^{n-1}$ then it is natural to introduce its dual gauge $H^\circ$, as in \cite{AFTL}.
It is defined by
\[H^\circ(z)=\sup_{H(y)\le 1} z\cdot y.\]
Then, the last condition on $\sigma$ in the definition of $P_H(\cdot)$ is equivalent to $H^\circ(\sigma)\leq1$ in $\R^n$,
and the Wulff shape can be written as $W=\{H^\circ<1\}$.

Some typical examples of gauges are
\[H(x)=\|x\|_p=\bigl(|x_1|^p+\cdots+|x_n|^p\bigr)^{1/p},\qquad 1\leq p\leq \infty.\]
Then, we have that $W=\{x\in\mathbb R^n:\ \|x\|_{p'}<1\}$, where ${p'}$ is such that $\frac1p+\frac{1}{p'}=1$.

The following is the celebrated Wulff inequality.

\begin{thm}[\cite{W,T1,T2}]\label{wulffthm}
Let $H$ be a gauge in $\R^n$ which is positive on $S^{n-1}$, and let $W$ be its associated Wulff shape.
Then, for every measurable set $E\subset \R^n$ with $|E|<\infty$, we have
\begin{equation}\label{ppp}
\frac{P_{H}(E)}{|E|^{\frac{n-1}{n}} }\ge \frac{P_{H}(W)}{|W|^{\frac{n-1}{n}}}.
\end{equation}
Moreover, equality holds if and only if $E=aW+b$ for some $a>0$ and $b\in\R^n$ except for a set of measure zero.
\end{thm}

This result was first stated without proof by Wulff \cite{W} in 1901.
His work was followed by Dinghas \cite{D}, who studied the problem within the class of convex polyhedra.
He used the Brunn-Minkowski inequality.
Some years later, Taylor \cite{T1,T2} finally proved Theorem \ref{wulffthm} among sets of finite perimeter;
see \cite{T3,FonMu,Mc} for more information on this topic.
As a particular case of our technique, in this paper we provide a new proof of
Theorem~\ref{wulffthm}. It is based on the ABP method applied to a linear Neumann problem.
It was Robert McCann (in a personal communication around 2000) who pointed out that the
first author's proof of the classical isoperimetric inequality
also worked in the anisotropic case.

\medskip
\subsection{\sc Results}
\hspace{\fill}
\smallskip

The main result of the present paper, Theorem \ref{th1} below,
is a weighted isoperimetric inequality which extends the two previous classical
inequalities (Theorems \ref{isopcon} and \ref{wulffthm}).
In particular, in Section \ref{sec2}
we will give a new proof of the classical Wulff theorem (for smooth domains) using the ABP method.

Before stating our main result, let us define the weighted anisotropic
perimeter relative to an open cone $\Sigma$.
The weights $w$ that we consider will always be continuous functions in $\overline\Sigma$, positive and locally Lipschitz in $\Sigma$,
and homogeneous of degree $\alpha\geq0$.
Given a gauge $H$ in $\R^n$ and a weight $w$, we define (following \cite{BBF}) the weighted anisotropic perimeter relative to the cone $\Sigma$ by
\[P_{w,H}(E;\Sigma) := \sup \biggl\{\int_{E\cap \Sigma}\hspace{-3mm} \div(\sigma w)dx\,:\, \sigma\in X_{w,\Sigma}\,,\sup_{H(y)\le 1} (\sigma(x)\cdot y)\le 1\ \textrm{for}\ x\in\Sigma \biggr\},\]
where $E\subset\R^n$ is any measurable set with finite Lebesgue measure and
\[X_{w,\Sigma} := \left\{ \sigma\in \bigl(L^\infty(\Sigma)\bigr)^n \,:\,{\rm div}(\sigma w)\in L^\infty \bigl(\Sigma\bigr) \ \ {\rm and}\ \ \sigma w =0 \ \,{\rm on}\ \,\partial \Sigma \right\}.\]
It is not difficult to see that
\begin{equation} \label{defperiint}
P_{w,H}(E;\Sigma) = \int_{\partial E \cap \Sigma}H\bigl(\nu(x)\bigr)w(x)dS\end{equation}
whenever $E$ is smooth enough.

The definition of $P_{w,H}$ is the same as the one given in \cite{BBF}.
In their notation, we are taking $d\mu = w\chi_\Sigma \,dx$ and $X_{w,\Sigma}=X_\mu$.

Moreover, when $H$ is the Euclidean norm we denote
\[P_w(E;\Sigma):=P_{w,\|\cdot\|_2}(E;\Sigma).\]
When $w\equiv1$ in $\Sigma$ and $H$ is the Euclidean norm we recover the definition of $P(E;\Sigma)$; see \cite{BBF}.

Given a measurable set $F\subset \Sigma$, we denote by $w(F)$ the weighted volume of $F$
\[ w(F):= \int_F w\,dx .\]
We also denote
\[D=n+\alpha.\]
Note that the Wulff shape $W$ is independent of the weight $w$.
Next we use that if $\nu$ is the unit outward normal to $W\cap \Sigma$, then $x\cdot\nu(x)=H(\nu(x))$ a.e. on $\partial W\cap \Sigma$, $x\cdot\nu(x)=0$ a.e. on $\overline W\cap\partial\Sigma$, and $x\cdot \nabla w(x) = \alpha w(x)$ in $\Sigma$.
This last equality follows from the homogeneity of degree $\alpha$ of $w$.
Then, with a similar argument as in \eqref{per/vol}, we find
\begin{equation}\label{formula-per=Dvol-for-Wulff}
\begin{split} P_{w,H}(W;\Sigma)&= \int_{\partial W\cap \Sigma} H(\nu(x))w(x)dS= \int_{\partial W\cap \Sigma} x\cdot\nu(x)\,w(x)dS\\
&=\int_{\partial(W\cap \Sigma)}x\cdot\nu(x)w(x)dS= \int_{W\cap\Sigma} \div(x w(x))dx\\
&=\int_{W\cap\Sigma} \left\{nw(x)+x\cdot\nabla w(x)\right\}dx= D\,w(W\cap \Sigma).
\end{split}
\end{equation}
Here ---and in our main result that follows--- for all quantities to make sense we need to assume that $W\cap \Sigma\neq \varnothing$.
Recall that both $W$ and $\Sigma$ are open convex sets but that $W\cap \Sigma= \varnothing$ could happen.
This occurs for instance if $H|_{S^{n-1}\cap\Sigma}\equiv0$.
On the other hand, if $H>0$ on all $S^{n-1}$ then $W\cap \Sigma\neq\varnothing$.

The following is our main result.

\begin{thm}\label{th1}
Let $H$ be a gauge in $\R^n$, i.e., a function satisfying \eqref{gauge}, and $W$ its associated Wulff shape defined by \eqref{wulffshape}.
Let $\Sigma$ be an open convex cone in $\mathbb R^n$ with vertex at the origin, and such that
$W\cap \Sigma\neq \varnothing$.
Let $w$ be a continuous function in $\overline\Sigma$, positive in $\Sigma$, and positively homogeneous of degree $\alpha\geq0$.
Assume in addition that $w^{1/\alpha}$ is concave in $\Sigma$ in case $\alpha>0$.

Then, for each measurable set $E\subset\R^n$ with $w(E\cap \Sigma)<\infty$,
\begin{equation}\label{mainresult}
\frac{P_{w,H}(E;\Sigma) }{w(E\cap \Sigma)^{\frac{D-1}{D}} }\geq \frac{P_{w,H}(W;\Sigma) }{w(W\cap \Sigma)^{\frac{D-1}{D}}},
\end{equation}
where $D=n+\alpha$.
\end{thm}

\begin{rem}\label{villani}
Our key hypothesis that $w^{1/\alpha}$ is a concave function is equivalent to a
natural curvature-dimension bound (in fact, to the nonnegativeness of the
Bakry-\'Emery Ricci tensor in dimension $D=n+\alpha$). This was suggested to us by C\'edric Villani, and
has also been noticed by Ca\~nete and Rosales (see Lemma~3.9 in \cite{CaRo}).
More precisely, we see the cone $\Sigma\subset\R^n$ as a Riemannian manifold of dimension $n$ equipped
with a reference measure $w(x)dx$. We are also given a ``dimension'' $D=n+\alpha$.
Consider the Bakry-\'Emery Ricci tensor, defined by
$$
\text{Ric}_{D,w} = \text{Ric} - \nabla^2 \log w - \frac{1}{D-n} \nabla \log w \otimes \nabla \log w.
$$
Now, our assumption $w^{1/\alpha}$ being concave is equivalent to
\begin{equation}\label{Bakry}
\text{Ric}_{D,w} \geq 0.
\end{equation}
Indeed, since $\text{Ric}\equiv 0$ and $D-n=\alpha$, \eqref{Bakry} reads as
$$
- \nabla^2 \log w^{1/\alpha} - \nabla \log w^{1/\alpha} \otimes \nabla \log w^{1/\alpha} \geq 0,
$$
which is the same condition as $w^{1/\alpha}$ being concave. Condition \eqref{Bakry} is called
a curvature-dimension bound; in the terminology of \cite{V} we say that
$\text{CD}(0,D)$ is satisfied by $\Sigma\subset\R^n$ with the reference measure $w(x)dx$.

In addition, C. Villani pointed out that optimal transport techniques could also lead
to weighted isoperimetric inequalities in convex cones; see Section~\ref{proof-related}.
\end{rem}

Note that the shape of the minimizer is $W\cap\Sigma$, and that $W$ depends only on $H$ and not on the weight $w$ neither on the cone $\Sigma$.
In particular, in the isotropic case $H=\|\cdot\|_{2}$ we find the following surprising fact.
Even that the weights that we consider are not radial (unless $w\equiv1$), still Euclidean balls centered at the origin (intersected with the cone) minimize this isoperimetric quotient.
The only explanation that one has a priori for this fact is that Euclidean balls centered at 0 have constant generalized mean curvature when the weight is homogeneous, as pointed out in \eqref{genmean}.
Thus, they are candidates to minimize perimeter for a given volume.

Note also that we allow $w$ to vanish somewhere (or everywhere) on $\partial\Sigma$.

Equality in \eqref{mainresult} holds whenever $E\cap \Sigma=rW\cap \Sigma$, where $r$ is any positive number.
However, in this paper we do not prove that $W\cap\Sigma$ is the unique
minimizer of \eqref{mainresult}.
The reason is that our proof involves the solution of an elliptic
equation and, due to an important issue on its regularity, we need to
approximate the given set $E$ by smooth sets.
In a future work with E. Cinti and A. Pratelli we will refine the analysis in the proof
of the present article and obtain a quantitative version of our
isoperimetric inequality in cones. In particular, we will deduce uniqueness of minimizers
(up to sets of measure zero).
The quantitative version will be proved using the techniques of the present paper (the ABP method
applied to a linear Neumann problem) together with the ideas of Figalli-Maggi-Pratelli~\cite{FMP-quantitative}.

In the isotropic case, a very recent result of Ca\~nete and Rosales \cite{CaRo} deals with
the same class of weights as ours. They allow not only positive homogeneities $\alpha>0$, but also negative 
ones $\alpha \leq -(n-1)$.
They prove that if a smooth, compact, and orientable hypersurfaces in a smooth convex cone is  stable 
for weighted perimeter (under variations preserving weighted volume), then it must be a sphere centered at 
the vertex of the cone.
In \cite{CaRo} the stability of such spheres is proved for $\alpha \leq -(n-1)$, but not for
$\alpha >0$.  However, as pointed out in \cite{CaRo}, when $\alpha >0$ their result used together 
with ours give 
that spheres centered at the vertex are the unique minimizers among smooth hypersurfaces.

Theorem \ref{th1} contains the classical isoperimetric inequality, its version for convex cones, and the classical Wulff inequality.
Indeed, taking $w\equiv 1$, $\Sigma=\R^n$, and $H = \|\cdot\|_{2}$ we recover the classical isoperimetric inequality with optimal constant.
Still taking $w\equiv 1$ and $H=\|\cdot\|_2$ but now letting $\Sigma$  be any open convex cone of $\R^n$ we have the isoperimetric inequality in convex cones of Lions and Pacella (Theorem \ref{isopcon}).
Moreover, if we take $w\equiv 1$ and $\Sigma=\R^n$ but we let $H$ be some other gauge we obtain the Wulff inequality (Theorem \ref{wulffthm}).

A criterion of concavity for homogeneous functions of degree 1 can be found for example in \cite[Proposition 10.3]{M-Llibre}, and reads as follows.
A nonnegative, $C^2$, and homogeneous of degree 1 function $\Phi$ on $\R^n$ is concave if and only if the restrictions $\Phi(\theta)$ of $\Phi$ to one-dimensional circles about the origin satisfy
\[\Phi''(\theta)+\Phi(\theta)\leq0.\]
Therefore, it follows that a nonnegative, $C^2$, and homogeneous weight of degree $\alpha>0$ in the plane $\R^2$, $w(x)=r^\alpha B(\theta)$, satisfies that $w^{1/\alpha}$ is concave in $\Sigma$ if and only if
\[(B^{1/\alpha})''+B^{1/\alpha}\leq0.\]

\begin{rem}\label{stable}
Let $w$ be an homogeneous weight of degree $\alpha$, and consider the isotropic isoperimetric problem in a cone $\Sigma\subset\R^n$.
Then, by the proofs of Proposition 3.6 and Lemma 3.8 in \cite{RCBM} the set $B_1(0)\cap \Sigma$ is stable if and only if
\begin{equation}\label{stable}
\int_{ S^{n-1}\cap\Sigma} |\nabla_{S^{n-1}} u|^2w\,dS\geq (n-1+\alpha)\int_{ S^{n-1}\cap\Sigma} |u|^2w\,dS
\end{equation}
for all functions $u\in C^\infty_c(S^{n-1}\cap \Sigma)$ satisfying
\begin{equation}\label{constraint}
\int_{ S^{n-1}\cap\Sigma}uw\,dS=0.\end{equation}
Stability being a necessary condition for minimality, from Theorem \eqref{th1} we deduce the following.
If $\alpha>0$, $\Sigma$ is convex, and $w^{1/\alpha}$ is concave in $\Sigma$, then \eqref{stable} holds.

For instance, in dimension $n=2$, inequality \eqref{stable} reads as
\begin{equation}\label{wirtin}
\int_0^\beta (u')^2w\,d\theta\geq (1+\alpha)\int_0^\beta u^2w\,d\theta\qquad {\rm whenever}\quad\ \int_0^\beta uw\,d\theta=0,\end{equation}
where $0<\beta\leq\pi$ is the angle of the convex cone $\Sigma\subset\R^2$.
This is ensured by our concavity condition on the weight $w$,
\begin{equation}\label{concav}
\left(w^{1/\alpha}\right)''+w^{1/\alpha}\leq 0\qquad {\rm in}\ (0,\beta).\end{equation}
Note that, even in this two-dimensional case, it is not obvious that this condition on $w$ yields 
\eqref{stable}-\eqref{constraint}.
The statement \eqref{wirtin} is an extension of Wirtinger's inequality (which corresponds to the case $w\equiv1$, $\alpha=0$, $\beta=2\pi$).
It holds, for example, with $w=\sin^\alpha\theta$ on $S^1$ ---since \eqref{concav} is satisfied by this weight.
Another extension of Wirtinger's inequality (coming from the density $w=r^\alpha$) is given in \cite{DDNT}.
\end{rem}

In Theorem \ref{th1} we assume that $w$ is homogeneous of degree $\alpha$.
In our proof, this assumption is essential in order that the paraboloid in \eqref{parabola} solves the PDE in
\eqref{pb_neu_pes}, as explained in Section \ref{elementsproof}.
Due to the homogeneity of $w$, the exponent $D=n+\alpha$  can be found just by a scaling argument in our inequality \eqref{mainresult}.
Note that this exponent $D$ has a dimension flavor if one compares \eqref{mainresult} with \eqref{pp} or with \eqref{ppp}.
Also, it is the exponent for the volume growth, in the sense that $w(B_r(0)\cap \Sigma)=Cr^D$ for all $r>0$.
The interpretation of $D$ as a dimension is more clear in the following example that motivated our work.

\begin{rem} \label{remmon}
The monomial weights
\begin{equation}\label{monomial}
w(x)=x_1^{A_1}\cdots x_n^{A_n}\qquad \textrm{in}\quad \Sigma=\{x\in\mathbb R^n\,:\, x_i>0
\textrm{ whenever }A_i>0\},\end{equation}
where $A_i\geq0$, $\alpha=A_1+\cdots+A_n$, and $D=n+A_1+\cdots+A_n$, are important examples for which \eqref{mainresult} holds.
The isoperimetric inequality ---and the corresponding Sobolev inequality \eqref{Sob}--- with the above monomial weights were studied by the first two authors in \cite{CR,CR2}.
These inequalities arose in \cite{CR} while studying reaction-diffusion problems with symmetry of double revolution.
A function $u$ has symmetry of double revolution when $u(x,y)= u(|x|,|y|)$, with  $(x,y) \in \mathbb R^{D}= \mathbb R^{A_1+1}\times\mathbb R^{A_2+1}$
(here we assume $A_i$ to be positive integers).
In this way, $u=u(x_1,x_2)=u(|x|,|y|)$ can be seen as a function in $\mathbb R^2=\mathbb R^n$, and it is here where the Jacobian for the Lebesgue measure in $\R^D=\R^{A_1+1}\times\R^{A_2+1}$, $x_1^{A_1}x_2^{A_2}=|x|^{A_1}|y|^{A_2}$, appears.
A similar argument under multiple axial symmetries shows that, when $w$ and $\Sigma$ are given by (\ref{monomial}) and all $A_i$ are nonnegative integers, and $H$ is the Euclidean norm, Theorem \ref{th1} follows from the classical isoperimetric inequality in $\mathbb R^D$; see \cite{CR2} for more details.

In \cite{CR} we needed to show a Sobolev inequality of the type \eqref{Sob} in $\R^2$ with the weight and the cone given by \eqref{monomial}.
As explained above, when $A_i$ are all nonnegative integers this Sobolev inequality follows from the classical one in dimension $D$.
However, in our application the exponents $A_i$ were not integers ---see \cite{CR}---, and thus the Sobolev inequality was not known.
We showed a nonoptimal version (without the best constant) of that Sobolev inequality in dimension $n=2$ in \cite{CR}, and later we proved in \cite{CR2} the optimal one in all dimensions $n$, obtaining the best constant and extremal functions for the inequality.
In both cases, the main tool to prove these Sobolev inequalities was an isoperimetric inequality with the same weight.
\end{rem}

An immediate consequence of Theorem \ref{th1} is the following weighted isoperimetric inequality in $\R^n$ for symmetric sets and even weights.
It follows from our main result taking $\Sigma=(0,+\infty)^n$.

\begin{cor}\label{cor1}
Let $w$ be a nonnegative continuous function in $\R^n$, even with respect to each variable, homogeneous of degree $\alpha>0$,
and such that $w^{1/\alpha}$ is concave in $(0,\infty)^n$.
Let $E\subset\R^n$ be any measurable set, symmetric with respect to each coordinate hyperplane $\{x_i=0\}$, and with $|E|<\infty$.
Then,
\begin{equation}\label{mainresultcor}
\frac{P_{w}(E;\R^n) }{|E|^{\frac{D-1}{D}} }\geq \frac{P_{w}(B_1;\R^n) }{|B_1|^{\frac{D-1}{D}}},
\end{equation}
where $D=n+\alpha$ and $B_1$ is the unit ball in $\R^n$.
\end{cor}

The symmetry assumption on the sets that we consider in Corollary \ref{cor1} is satisfied in some applications arising in nonlinear problems, such as the one in \cite{CR} explained in Remark \ref{remmon}.
Without this symmetry assumption, isoperimetric sets in \eqref{mainresultcor} may not be the balls.
For example, for the monomial weight $w(x)=|x_1|^{A_1}\cdots |x_n|^{A_n}$ in $\R^n$, with all $A_i$ positive, $B_1\cap (0,\infty)^n$ is an isoperimetric set,
while the whole ball $B_r$ having the same weighted volume as $B_1\cap (0,\infty)^n$ is not an isoperimetric set (since it has longer perimeter).

We know only of few results where nonradial weights lead to radial minimizers.
The first one is the isoperimetric inequality by Maderna-Salsa \cite{MS} in the upper half plane $\R^2_+$ with the weight $x_2^\alpha$, $\alpha>0$.
To establish their isoperimetric inequality, they first proved the existence of a minimizer for the
perimeter functional under constraint of fixed area, then computed the first variation
of this functional, and finally solved the obtained ODE to find all minimizers.
The second result is due to Brock-Chiacchio-Mercaldo \cite{BCM} and extends the one in \cite{MS} by including the weights $x_n^{\alpha}\exp(c|x|^2)$ in $\mathbb R^n_+$, with $\alpha\geq0$ and $c\geq0$.
In both papers it is proved that half balls centered at the origin are the minimizers of the isoperimetric quotient with these weights.
Another one, of course, is our isoperimetric inequality with monomial weights \cite{CR2} explained above (see Remark \ref{remmon}).
At the same time as us, and using totally different methods, Brock, Chiacchio, and Mercaldo \cite{BCM2} have proved an isoperimetric inequality in $\Sigma=\{x_1>0,...,x_n>0\}$ with the weight $x_1^{A_1}\cdots x_n^{A_n}\exp(c|x|^2)$, with $A_i\geq0$ and $c\geq0$.

In all these results, although the weight $x_1^{A_1}\cdots x_n^{A_n}$ is not radial, it has a very special structure.
Indeed, when all $A_1,...,A_n$ are nonnegative \emph{integers} the isoperimetric problem
with the weight $x_1^{A_1}\cdots x_n^{A_n}$ is equivalent to the isoperimetric problem in $\R^{n+A_1+\cdots+A_n}$ for
sets that have symmetry of revolution with respect to the first $A_1+1$ variables,
the next $A_2+1$ variables, ..., and so on until the last $A_n+1$ variables; see Remark \ref{remmon}.
By this observation, the fact that half balls centered at the origin are the minimizers in $\R^{n}_+$ with the weight $x_1^{A_1}\cdots x_n^{A_n}$ or $x_1^{A_1}\cdots x_n^{A_n}\exp(c|x|^2)$, for $c\geq0$ and $A_i$ nonnegative integers, follows from the isoperimetric inequality in $\R^{n+A_1+\cdots+A_n}$ with the weight $\exp(c|x|^2)$, $c\geq0$ (which is a radial weight).
Thus, it was reasonable to expect that the same result for noninteger exponents $A_1,...,A_n$ would also hold ---as it does.

After announcing our result and proof in \cite{CRS-CRAS},
Emanuel Milman showed us a nice geometric
construction that yields the particular case when $\alpha$ is a
nonnegative  {\it integer} in our weighted inequality of Theorem \ref{th1}.
Using this construction, the weighted inequality in a convex cone is obtained as a limit case of the
unweighted Lions-Pacella inequality in a narrow cone of $\R^{n+\alpha}$.
We reproduce it in Remark \ref{milmans} ---see also the blog of Frank Morgan \cite{Mblog}.

\medskip
\subsection{\sc The proof. Related works}
\label{proof-related}
\hspace{\fill}
\smallskip

The proof of Theorem \ref{th1} consists of applying the ABP method to a linear Neumann problem involving the operator $w^{-1}\textrm{div}(w\nabla u)$, where $w$ is the weight.
When $w\equiv 1$, the idea goes back to 2000 in the works \cite{CSCM,CDCDS} of the first author, where the classical isoperimetric inequality in all of $\R^n$ (here $w\equiv1$) was proved with a new method.
It consisted of solving the problem
\[
\left\{ \alignedat2
\Delta u &= b_\Omega  &\quad &\mbox{in } \Omega
\\
\frac{\partial u}{\partial\nu} &=1 &\quad &\textrm{on }\partial\Omega
\endalignedat\right. \]
for a certain constant $b_\Omega$, to produce a bijective map with the gradient of $u$, $\nabla u:\Gamma_{u,1}\longrightarrow B_1$, which leads to the isoperimetric inequality.
Here $\Gamma_{u,1}\subset\Gamma_u\subset\Omega$ and $\Gamma_{u,1}$ is a certain subset of the lower contact set $\Gamma_u$ of $u$ (see Section \ref{elementsproof} for details).
The use of the ABP method is crucial in the proof.

Previously, Trudinger \cite{Trudinger} had given a proof of the classical isoperimetric inequality in 1994 using the theory of Monge-Amp\`ere equations and the ABP estimate.
His proof consists of applying the ABP estimate to the Monge-Amp\`ere problem
\[
\left\{ \alignedat2
 \textrm{det}D^2u &= \chi_\Omega  &\quad &\mbox{in } B_R
\\
u &=0 &\quad &\textrm{on }\partial B_R,
\endalignedat\right. \]
where $\chi_\Omega$ is the characteristic function of $\Omega$ and $B_R=B_R(0)$, and then letting $R\rightarrow\infty$.

Before these two works (\cite{Trudinger} and \cite{CSCM}), there was already a proof of the isoperimetric inequality using a certain map (or coupling).
This is Gromov's proof, which used the Knothe map; see \cite{V}.

After these three proofs, in 2004 Cordero-Erausquin, Nazaret, and Villani \cite{CNV} used the Brenier map
from optimal transportation to give a beautiful proof of the anisotropic isoperimetric inequality; see also \cite{V}.
More recently, Figalli-Maggi-Pratelli \cite{FMP-quantitative} established a sharp quantitative version of the anisotropic isoperimetric inequality, using also the Brenier map.
In the case of the Lions-Pacella isoperimetric inequality, this has been done by Figalli-Indrei \cite{FI} very recently. 
As mentioned before, the proof in the present article
is also suited for a quantitative version, as we will show in a future work with Cinti and Pratelli.

After announcing our result and proof in \cite{CRS-CRAS}, we have been told that
optimal transportation techniques \`a la \cite{CNV}
could also be used to prove weighted isoperimetric inequalities
in certain cones. C. Villani pointed out that this is mentioned in the Bibliographical Notes to Chapter~21
of his book \cite{V}. A. Figalli showed it to us with a computation
when the cone
is a halfspace $\{x_n>0\}$ equipped with
the weight $x_n^\alpha$.

\medskip
\subsection{\sc Applications}
\hspace{\fill}
\smallskip

Now we turn to some applications of Theorems \ref{th1} and Corollary \ref{cor1}.

First, our result leads to weighted Sobolev inequalities with best constant in convex cones of $\R^n$.
Indeed, given any smooth function $u$ with compact support in $\R^n$
(we do not assume $u$ to vanish on $\partial\Sigma$), one uses the coarea formula and
Theorem \ref{th1} applied to each of the level sets of $u$.
This establishes the Sobolev inequality \eqref{Sob} for $p=1$.
The constant $C_{w,1,n}$ obtained in this way is optimal, and coincides with the best constant in our
isoperimetric inequality \eqref{mainresultcor}.

When $1<p<D$, Theorem \ref{th1} also leads to the
Sobolev inequality \eqref{Sob} with best constant. This is a consequence of our isoperimetric
inequality and a weighted radial rearrangement
of Talenti \cite{T}, since these two results yield the radial symmetry of minimizers.
See \cite{CR2} for details in the case of monomial weights
$w(x)=|x_1|^{A_1}\cdots |x_n|^{A_n}$.

If we use Corollary \ref{cor1} instead of Theorem \ref{th1}, with the same argument one finds the Sobolev
inequality
\begin{equation}\label{Sob2}
\left(\int_{\R^n}|u|^{p_*}w(x)dx\right)^{1/p_*}\leq C_{w,p,n}\left(\int_{\R^n}|\nabla u|^pw(x)dx\right)^{1/p},
\end{equation}
where $p_*=\frac{pD}{D-p}$, $D=n+\alpha$, and $1\leq p<D$.
Here, $w$ is any weight satisfying the hypotheses of Corollary \ref{cor1}, and $u$ is any smooth function
with compact support in $\R^n$ which is symmetric with respect to each variable $x_i$, $i=1,...,n$.

We now turn to applications to the symmetry of solutions to nonlinear PDEs.
It is well known that the classical isoperimetric inequality yields some radial symmetry results for semilinear or quasilinear elliptic equations.
Indeed, using the Schwartz rearrangement that preserves $\int F(u)$ and decreases $\int \Phi(|\nabla u|)$, it is immediate to show that minimizers of some energy functionals (or quotients) involving these quantities are radially symmetric; see \cite{PS,T}.
Moreover, P.-L. Lions \cite{L} showed that in dimension $n=2$ the isoperimetric inequality yields also the radial symmetry of all positive solutions
to the semilinear problem $-\Delta u=f(u)$ in $B_1$, $u=0$ on $\partial B_1$, with $f\geq0$ and $f$ possibly discontinuous.
This argument has been extended in three directions: for the $p$-Laplace operator, for cones of $\R^n$, and for Wulff shapes, as explained next.

On the one hand, the analogue of Lions radial symmetry result but in dimension $n\geq3$ for the $p$-Laplace operator was proved with $p=n$ by Kesavan and Pacella in \cite{KP},
and with $p\geq n$ by the third author in \cite{S}.
Moreover, in \cite{KP} it is also proved that positive solutions to the following semilinear equation with mixed boundary conditions
\begin{equation}\label{mixed}
\left\{ \alignedat2
 -\Delta_p u &= f(u)  &\quad &\mbox{in } B_1\cap\Sigma
\\
u &= 0  &\quad &\mbox{on } \partial B_1\cap\Sigma
\\
\frac{\partial u}{\partial\nu} &=0 &\quad &\textrm{on }B_1\cap\partial\Sigma
\endalignedat\right. \end{equation}
have radial symmetry whenever $p=n$.
Here, $B_1$ is the unit ball and $\Sigma$ any open convex cone.
This was proved by using Theorem \ref{isopcon} and the argument of P.-L. Lions mentioned above.

On the other hand, Theorem \ref{wulffthm} is used to construct a Wulff shaped rearrangement in \cite{AFTL}.
This yields that minimizers to certain nonlinear variational equations that come from anisotropic gradient norms have Wulff shaped level sets.
Moreover, the radial symmetry argument in \cite{L} was extended to this anisotropic case in \cite{BFK}, yielding the same kind of result for positive solutions of nonlinear equations involving the operator $Lu={\rm div}\left(H(\nabla u)^{p-1}\nabla H(\nabla u)\right)$ with $p=n$.
In the same direction, in a future paper \cite{RS} we will use Theorem \ref{th1} to obtain Wulff shaped
symmetry of critical functions of weighted anisotropic functionals such as
\[\int\bigl\{ H^p(\nabla u)- F(u)\bigr\}w(x)\,dx.\]
Here, $w$ is an homogeneous weight satisfying the hypotheses of Theorem \ref{th1} and $H$ is any norm in $\R^n$.
As in \cite{S}, we will allow $p\neq n$ but with some conditions on $F$ in case $p<n$.

Related to these results, when $f$ is Lipschitz, Berestycki and Pacella \cite{BP} proved that any positive solution to problem \eqref{mixed} with $p=2$ in a convex spherical sector $\Sigma$ of $\R^n$ is radially symmetric.
They used the moving planes method.

\medskip
\subsection{\sc Plan of the paper}
\hspace{\fill}
\smallskip

The rest of the article is organized as follows.
In Section \ref{examples} we give examples of weights for which our result applies.
In Section \ref{elementsproof} we introduce the elements appearing in the proof of Theorem \ref{th1}.
To illustrate these ideas, in Section \ref{sec2} we give the proof of the classical Wulff theorem via the ABP method.
In Section \ref{sec3} we prove Theorem \ref{th1} in the simpler case $w\equiv0$ on $\partial\Sigma$ and $H=\|\cdot\|_2$.
Finally, in Section \ref{sec4} we present the whole proof of Theorem \ref{th1}.

\bigskip
\section{Examples of weights}
\label{examples}

When $w\equiv 1$ our main result yields the classical isoperimetric inequality, its version for convex cones, and also the Wulff theorem.
On the other hand, given an open convex cone $\Sigma\subset \R^n$ (different than the whole space and a half-space) there is a large family of functions that are homogeneous of degree one and concave in $\Sigma$.
Any positive power of one of these functions is an admissible weight for Theorem \ref{th1}.
Next we give some concrete examples of weights $w$ for which our result applies.
The key point is to check that the homogeneous function of degree one $w^{1/\alpha}$ is concave.

\begin{enumerate}
\item[(i)] Assume that $w_1$ and $w_2$ are concave homogeneous functions
of degree one in an open convex cone $\Sigma$.
    Then, $w_1^{a}w_2^{b}$ with $a\geq0$ and $b\geq0$, $(w_1^{r}+w_2^{r})^{\alpha/r}$ with $r\in(0,1]$ or $r<0$, and $\min\{w_1,w_2\}^\alpha$, satisfy the hypotheses of Theorem \ref{th1} (with $\alpha=a+b$ in the first case).
    More generally, if $F:[0,\infty)^2\rightarrow \R_+$ is positive, concave, homogeneous of degree 1, and nondecreasing in each variable, then one can take $w=F(w_1,w_2)^\alpha$, with $\alpha>0$.
\item[(ii)] The distance function to the boundary of any convex set is concave when defined in the convex set.
    On the other hand, the distance function to the boundary of any cone is homogeneous of degree 1.
    Thus, for any open convex cone $\Sigma$ and any $\alpha\geq0$,
    \[w(x)=\textrm{dist}(x,\partial\Sigma)^\alpha\]
    is an admissible weight.
    When the cone is $\Sigma=\{x_i>0,\ i=1,...,n\}$, this weight is exactly $\min\{x_1,...,x_n\}^\alpha$.
\item[(iii)] If the concavity condition is satisfied by a weight $w$ in a convex cone $\Sigma'$ then it is also satisfied in any convex subcone $\Sigma\subset \Sigma'$.
    Note that this gives examples of weights $w$ and cones $\Sigma$ in which $w$ is positive on $\partial\Sigma\setminus\{0\}$.
\item[(iv)] Let $\Sigma_1,...,\Sigma_k$ be convex cones and $\Sigma=\Sigma_1\cap\cdots\cap \Sigma_k$.
    Let
    \[\delta_i(x)={\rm dist}(x,\partial \Sigma_i).\]
    Then, the weight
    \[w(x)=\delta_1^{A_1}\cdots\delta_k^{A_k},\ \ x\in\Sigma,\]
    with $A_1\geq0,...,A_k\geq0$, satisfies the hypotheses of Theorem \ref{th1}.
    This follows from (i), (ii), and (iii).
    Note that when $k=n$ and $\Sigma_{i}=\{x_{i}>0\}$, $i=1,...,n$, then $\Sigma=\{x_1>0,...,x_n>0\}$ and we obtain the monomial weight
    \[w(x)=x_1^{A_1}\cdots x_n^{A_n}.\]

\item[(v)] In the cone $\Sigma=(0,\infty)^n$, the weights
    \[w(x)=\left(A_1x_1^{1/p}+\cdots+A_nx_n^{1/p}\right)^{\alpha p},\]
    for $p\geq1$, $A_i\geq0$, and $\alpha>0$, satisfy the hypotheses of Theorem \ref{th1}.
    Similarly, one may take the weights
    \[w(x)=\left({\frac{A_1}{x_1^{r}}+\cdots+\frac{A_n}{x_n^{r}}}\right)^{-\alpha/r},\]
    with $r>0$, or the limit case
    \[w(x)=\min\{A_1x_1,\cdots,A_nx_n\}^{\alpha}.\]
    This can be showed using the Minkowski inequality.
    More precisely, the first one can be showed using the classical Minkowski inequality with exponent $p\geq1$, while the second one using a reversed Minkowski inequality that holds for exponents $p=-r<0$.

    In these examples $\Sigma=(0,\infty)^n$ and therefore by Corollary \ref{cor1} we find that among all sets $E\subset\R^n$ which are symmetric with respect to each coordinate hyperplane, Euclidean balls centered at the origin minimize the isoperimetric quotient with these weights.
\item[(vi)] Powers of \emph{hyperbolic polynomials} also provide examples of weights.
    An homogeneous polynomial $P(x)$ of degree $k$ defined in $\R^n$ is called {hyperbolic} with respect to 
$a\in\R^n$ provided $P(a)>0$ and for every $\lambda\in \R^n$ the polynomial in $t$, $P(ta+\lambda)$, has exactly $k$ real roots.
    Let $\Sigma$ be the component in $\R^n$, containing $a$, of the set $\{P>0\}$.
    Then, $\Sigma$ is a convex cone and $P(x)^{1/k}$ is a concave function in $\Sigma$; see for example \cite{Garding2} or \cite[Section 1]{CNS}.
    Thus, for any hyperbolic polynomial $P$, the weight
    \[w(x)=P(x)^{\alpha/k}\]
    satisfies the hypotheses of Theorem \ref{th1}.
    Typical examples of hyperbolic polynomials are
    \[P(x)=x_1^2-\lambda_2x_2^2-\cdots-\lambda_nx_n^2\ \ \textrm{in}\ \Sigma=\left\{x_1>
\sqrt{\lambda_2x_2^2+\cdots+\lambda_nx_n^2}\right\},\]
    with $\lambda_2>0$,...,$\lambda_n>0$,
    or the elementary symmetric functions
    \[\sigma_k(x)=\sum_{1\leq i_1<\cdots<i_k\leq n}x_{i_1}\cdots x_{i_k}\quad \textrm{in}\ \Sigma=\{\sigma_1>0,...,\sigma_k>0\}\]
    (recall that $\Sigma$ is defined above as a component of $\{P>0\}$).
    Other examples are
    \[P(x)=\prod_{1\leq i_1<\cdots< i_r\leq n} \sum_{j=1}^r x_{i_j} \ \ \textrm{in}\
\Sigma=\{x_i>0,\ i=1,...,n\},\]
    which have degree $k={n\choose r}$ (this follows by induction from the first statement in example (i);
see also \cite{BGLS}),
or the polynomial $\det(X)$ in the convex cone of symmetric positive definite matrices ---which
we consider in the space $\R^{n(n+1)/2}$.

    The interest in hyperbolic polynomials was originally motivated by an important paper of Garding on linear hyperbolic PDEs \cite{Garding},
    and it is known that they form a rich class;
    see for example \cite{Garding2}, where the same author showed various ways of constructing new hyperbolic polynomials from old ones.
\item[(vii)] If $\sigma_k$ and $\sigma_l$ are the elementary symmetric functions of degree $k$ and $l$, with $1\leq k<l\leq n$, then $\left(\sigma_l/\sigma_k\right)^{\frac{1}{l-k}}$ is concave in the cone $\Sigma=\{\sigma_1>0,...,\sigma_k>0\}$; see \cite{ML}.
    Thus,
    \[w(x)=\left(\frac{\sigma_l}{\sigma_k}\right)^{\frac{\alpha}{l-k}}\]
    is an admissible weight.
    For example, setting $k=n$ and $l=1$ we find that we can take
    \[w(x)=\left(\frac{x_1\cdots x_n}{x_1+\cdots+x_n}\right)^{\frac{\alpha}{n-1}}\]
    in Theorem \ref{th1} or in Corollary \ref{cor1}.
\item[(viii)] If $f:\R\rightarrow\R_+$ is any continuous function which is concave in $(a,b)$, then
    \[w(x)=x_1f\left(\frac{x_2}{x_1}\right)\]
    is an admissible weight in $\Sigma=\{x=(r,\theta)\,:\, \arctan a<\theta<\arctan b\}$.
\item[(ix)] In the cone $\Sigma=(0,\infty)^2\subset\R^2$ one may take
    \[w(x)=\left(\frac{x_1-x_2}{\log x_1-\log x_2}\right)^\alpha\]
    for $\alpha>0$.
    In addition, in the same cone one may also take
    \[w(x)=\frac{1}{e}\left(x_1^{x_1}x_2^{-x_2}\right)^{\frac{\alpha}{x_1-x_2}}.\]
    This can be seen by using (viii) and computing $f$ in each of the two cases.
    When $\alpha=1$, these functions are called the \emph{logarithmic mean} and the \emph{identric mean} of the numbers $x_1$ and $x_2$, respectively.

    Using also (viii) one can check that, in the cone $\Sigma=(0,\infty)^2$, the weight $w(x)=xy(x^p+y^p)^{-1/p}$ is admissible whenever $p>-1$.
    Then, using (i) it follows that
    \[w(x)=\frac{x^{a+1}y^{b+1}}{(x^p+y^p)^{1/p}}\]
    is an admissible weight whenever $a\geq0$, $b\geq0$, and $p>-1$.
\end{enumerate}

\bigskip
\section{Description of the proof}
\label{elementsproof}

The proof of Theorem \ref{th1} follows the ideas introduced by the first author in a new proof of the classical isoperimetric inequality; see \cite{CSCM,CDCDS} or the last edition of Chavel's book \cite{Ch}.
This proof uses the ABP method, as explained next.

The Alexandroff-Bakelman-Pucci (or ABP) estimate
is an $L^\infty$ bound for solutions of the Dirichlet
problem associated to second order uniformly elliptic operators
written in nondivergence form,
\[Lu= a_{ij}(x) \partial_{ij} u +b_i (x) \partial_{i} u + c(x) u,\]
with bounded measurable coefficients in a domain
$\Omega$ of $\R^n$.
It asserts that if $\Omega$ is bounded and $c\leq 0$ in $\Omega$ then,
for every function $u\in C^2(\Omega)\cap C(\overline\Omega)$,
\[\sup_\Omega u\le \sup_{\partial\Omega} u
+ C \, {\rm diam} (\Omega)\, \Vert Lu \Vert_{L^n(\Omega)},\]
where $C$ is a constant depending only on the ellipticity constants of
$L$ and on the $L^n$-norm of the coefficients $b_i$.
The estimate was proven by the previous authors in the sixties using
a technique that is nowadays called ``the ABP method''.
See \cite{CDCDS} and references therein for more information on this estimate.

The proof of the classical isoperimetric inequality in \cite{CSCM,CDCDS}
consists of applying the ABP method
to an appropriate Neumann problem for the Laplacian
---instead of applying it to a Dirichlet problem as
customary.
Namely, to estimate from below $|\partial\Omega|/|\Omega|^{\frac{n-1}{n}}$ for a smooth domain $\Omega$, one considers the problem
\begin{equation} \label{eqlaplace}
\left\{ \alignedat2 \Delta u &= b_\Omega  &\quad &\mbox{in } \Omega
\\ \vspace{1mm}
\frac{\partial u}{\partial\nu} &=1 &\quad &\textrm{on }\partial \Omega.
\endalignedat\right. \end{equation}
The constant $b_\Omega = |\partial\Omega|/|\Omega|$ is chosen so that the problem has a solution.
Next, one proves that $B_1\subset \nabla u(\Gamma_u)$ with a contact argument (for a certain ``contact'' set $\Gamma_u\subset\Omega$),
and then one estimates the measure of $\nabla u(\Gamma_u)$
by using the area formula and the inequality between the geometric
and arithmetic means.
Note that the solution of \eqref{eqlaplace} is
\[u(x)=|x|^2/2\qquad \textrm{when}\ \Omega=B_1,\]
and in this case one verifies that
all the inequalities appearing in this ABP argument are equalities.
After having proved the isoperimetric inequality for smooth domains, an standard
approximation argument extends it to all sets of finite perimeter.

As pointed out by R. McCann, the same
method also yields the Wulff theorem.
For this, one replaces the Neumann data in \eqref{eqlaplace} by $\partial u/\partial\nu=H(\nu)$ and uses the same argument
explained above.
This proof of the Wulff theorem is given in Section~\ref{sec2}.

We now sketch the proof of Theorem \ref{th1} in the isotropic case, that is, when $H=\|\cdot\|_2$.
In this case, optimizers are Euclidean balls centered at the origin intersected with the cone.
First, we assume that $E=\Omega$ is a bounded smooth domain.
The key idea is to consider a similar problem to \eqref{eqlaplace}
but where the Laplacian is replaced by the operator
\[w^{-1}{\rm div}(w\nabla u)=\Delta u+\frac{\nabla w}{w}\cdot\nabla u.\]
Essentially (but, as we will see, this is not exactly as we proceed ---because of a regularity issue), we solve the following Neumann problem in $\Omega\subset\Sigma$:
\begin{equation} \label{pb_neu_pes}
\left\{
\alignedat2
w^{-1}{\rm div\,} (w\nabla u) &= b_\Omega&
\quad &\text{in } \Omega
\\ \vspace{1mm}
\frac{\partial u}{\partial\nu} &=1&\quad &\text{on }\partial \Omega \cap \Sigma
\\ \vspace{1mm}
\frac{\partial u}{\partial\nu} &= 0&\quad &\text{on }\partial \Omega \cap \partial\Sigma,
\endalignedat\right.
\end{equation}
where the constant $b_\Omega$ is again chosen depending on weighted perimeter and volume so that the problem admits a solution.
Whenever $u$ belongs to $C^1(\overline\Omega)$ ---which is not always the case, as discussed below in this section---, by touching
the graph of $u$ by below with planes (as in the proof of the classical isoperimetric inequality explained above) we find that
\begin{equation}\label{ABPset-inclusion}
B_1 \cap \Sigma \subset \nabla u\bigl(\Omega\bigr).
\end{equation}
Then,  using the area formula, an appropriate weighted
geometric-arithmetic means inequality, and the concavity condition on the weight $w$,
we obtain our weighted isoperimetric inequality.
Note that the solution of \eqref{pb_neu_pes} is
\begin{equation}\label{parabola}
u(x)=|x|^2/2\qquad \textrm{when}\ \Omega=B_1\cap \Sigma.\end{equation}
In this case, all the chain of inequalities in our proof become equalities, and this yields the sharpness of the result.

In the previous argument there is an important technical difficulty that comes from
the possible lack of regularity up to the boundary of the solution to the
weighted Neumann problem \eqref{pb_neu_pes}. For instance, if $\Omega\cap \Sigma$ is a smooth
domain that has some part of its boundary lying on $\partial \Sigma$ ---and hence $\partial \Omega$ meets tangentially $\partial \Sigma$---, then
$u$ can not be $C^1$ up to the boundary.
This is because the Neumann condition itself is not continuous and hence $\partial_\nu u$
would jump from $1$ to one $0$ where $\partial \Omega$ meets $\partial \Sigma$.

The fact that $u$ could not be $C^1$ up to the boundary prevents us from using our contact argument to prove
\eqref{ABPset-inclusion}.
Nevertheless, the argument sketched above does work  for smooth domains $\Omega$ well contained in $\Sigma$, that is, satisfying $\overline\Omega\subset\Sigma$.
If, in addition, $w\equiv0$ on $\partial\Sigma$ we can deduce the inequality for all measurable sets $E$ by an approximation argument.
Indeed, if $w\in C(\overline\Omega)$ and $w\equiv0$ on $\partial\Sigma$ then for any domain $U$ with piecewise Lipschitz boundary one has
\[ P_w(U;\Sigma) = \int_{\partial U\cap \Sigma} w\,dS = \int_{\partial U} w\,dS.\]
This fact allows us to approximate any set with finite measure $E\subset\Sigma$ by bounded smooth domains $\Omega_k$ satisfying $\overline{\Omega_k}\subset \Sigma$.
Thus, the proof of Theorem \ref{th1} for weights $w$ vanishing on $\partial\Sigma$ is simpler, and this is why we present it first, in Section \ref{sec3}.

Instead, if $w>0$ at some part of (or everywhere on) $\partial\Sigma$ it is not always possible to find sequences of smooth sets with closure contained in the open cone and approximating in relative perimeter a given measurable set $E\subset \Sigma$.
This is because the relative perimeter does not count the part of the boundary of $E$ which lies on $\partial\Sigma$.
To get around this difficulty (recall that we are describing the proof in the isotropic case, $H\equiv1$) we need to consider an
anisotropic problem in $\R^n$ for which approximation is possible.
Namely, we choose a gauge $H_0$ defined as the gauge associated to the convex set $B_1\cap\Sigma$; see \eqref{uniqueH}.
Then we prove that $P_{w,H_0}(\,\cdot\,;\Sigma)$ is a calibration
of the functional $P_{w}(\,\cdot\,;\Sigma)$, in the following sense.
For all $E\subset \Sigma$ we will have
\[P_{w,H_0}(E;\Sigma) \le P_{w}(E;\Sigma),\]
while for $E=B_1\cap \Sigma$,
\[P_{w,H_0}(B_1;\Sigma) = P_{w}(B_1\cap\Sigma;\Sigma).\]
As a consequence, the isoperimetric inequality with perimeter $P_{w,H_0}(\cdot;\Sigma)$ implies the one with the perimeter $P_w(\cdot;\Sigma)$.
For $P_{w,H_0}(\cdot;\Sigma)$  approximation results are available and, as in the case of $w\equiv 0$ on $\partial\Sigma$, it is enough to consider smooth sets satisfying $\overline\Omega\subset\Sigma$ ---for which there are no regularity problems with the solution of the elliptic problem.

To prove Theorem \ref{th1} for a general anisotropic perimeter $P_{w,H}(\cdot;\Sigma)$ we also consider a 
``calibrated''  perimeter $P_{w,H_0}(\cdot;\Sigma)$, where $H_0$ is now the gauge associated to the convex set $W\cap \Sigma$. Note that, as explained above, even for the isotropic case $H=\|\cdot\|_2$ we have to consider an anisotropic perimeter (associated to $B_1\cap \Sigma$) in order to prove Theorem \ref{th1}.

\bigskip
\section{Proof of the classical Wulff inequality}
\label{sec2}

In this section we prove the classical Wulff theorem for smooth domains by using the ideas introduced by the first author in \cite{CSCM,CDCDS}.
When $H$ is smooth on $S^{n-1}$, we show also that the Wulff shapes are the only smooth sets for which equality is attained.

\begin{proof}[Proof of Theorem \ref{wulffthm}]
We prove the Wulff inequality only for smooth domains, that we denote by $\Omega$ instead of $E$.
By approximation, if \eqref{ppp} holds for all smooth domains then it holds for all sets of finite perimeter.

By regularizing $H$ on $S^{n-1}$ and then extending it homogeneously, we can assume that $H$ is smooth in $\R^n\setminus\{0\}$.
For non-smooth $H$ this approximation argument will yield inequality \eqref{ppp}, but not the equality cases.

Let $u$ be a solution of the Neumann problem
\begin{equation}
\left\{
\alignedat2
\Delta u &= \frac{P_H(\Omega)}{\vert \Omega \vert}
&\quad &\text{in } \Omega\\
\frac{\partial u}{\partial\nu} &=H(\nu) &\quad &\text{on }\partial \Omega ,
\endalignedat
\right.
\label{eqsem}
\end{equation}
where $\Delta$ denotes the Laplace operator and $\partial u /\partial\nu$
the exterior normal derivative of $u$ on $\partial \Omega$.
Recall that $P_H(\Omega)=\int_{\partial \Omega} H\bigr(\nu(x)\bigl)\,dS$.
The constant $P_H(\Omega)/\vert \Omega \vert$
has been chosen so that the problem has a unique solution up to
an additive constant.
Since $H|_{S^{n-1}}$ and $\Omega$ are smooth, we have that $u$ is smooth in~$\overline \Omega$.
See \cite{N} for a recent exposition of these classical facts and for a new Schauder estimate for \eqref{eqsem}.

Consider the lower contact set of $u$, defined by
\begin{equation}
\Gamma_u =\left\{ x \in \Omega \ : \
u(y) \ge u(x) + \nabla u (x) \cdot (y-x)\  \text{ for all } y \in
\overline \Omega \right\} .
\label{lcset}
\end{equation}
It is the set of points where the tangent hyperplane to the
graph of $u$ lies below $u$ in all $\overline \Omega$. We claim that
\begin{equation}
W \subset \nabla u (\Gamma_u) ,
\label{gradmap}
\end{equation}
where $W$ denotes the Wulff shape associated to $H$, given by \eqref{wulffshape}.

To show \eqref{gradmap}, take any $p\in W$, i.e., any $p\in \R^n$ satisfying
\[p\cdot\nu<H(\nu)\ \ \textrm{for all}\ \nu\in S^{n-1}.\]
Let $x\in \overline \Omega$ be a point such that
$$
\min_{y\in \overline \Omega} \,\{ u(y)
-p\cdot y \} = u(x)-p\cdot x
$$
(this is, up to a sign, the Legendre transform of $u$).
If $x\in \partial \Omega$ then the exterior normal derivative of
$u(y)-p\cdot y$ at $x$ would be
nonpositive and hence $(\partial u /\partial\nu)
(x) \le p\cdot \nu < H(\nu)$, a contradiction with the boundary condition of \eqref{eqsem}.
It follows
that $x\in \Omega$ and, therefore, that $x$ is an interior minimum of
the function $u(y)-p\cdot y$. In particular,
$p=\nabla u (x)$ and $x\in \Gamma_u$. Claim \eqref{gradmap} is now proved.
It is interesting to visualize geometrically the proof
of the claim, by considering
the graphs of the functions $p\cdot y + c$ for $c\in \R$. These are
parallel hyperplanes which lie, for $c$ close to $-\infty$,
below the graph of $u$.
We let~$c$ increase and consider the first~$c$
for which there is contact or ``touching'' at a point~$x$.
It is clear geometrically that
$x\not\in\partial\Omega$,
since $p\cdot\nu <H(\nu)$ for all $\nu\in S^{n-1}$ and $\partial u /\partial\nu =H(\nu)$ on
$\partial\Omega$.

Now, from \eqref{gradmap} we deduce
\begin{equation}
\vert W\vert \le \vert \nabla u (\Gamma_u) \vert =
\int_{\nabla u (\Gamma_u)} dp
\le \int_{\Gamma_u} \det D^2u (x) \ dx .
\label{ineqwulff}
\end{equation}
We have applied the area formula to the smooth map $\nabla u : \Gamma_u
\rightarrow \R^n$, and we have used that its Jacobian,
$\det D^2u$, is nonnegative in $\Gamma_u$ by definition of this set.

Next, we use the classical inequality between the geometric and the arithmetic means
applied to the eigenvalues of $D^2u(x)$
(which are nonnegative numbers for $x\in \Gamma_u$). We obtain
\begin{equation}
\det D^2u \le \left( \frac{\Delta u}{n} \right)^n \quad \text{in }
\Gamma_u .
\label{means}
\end{equation}
This, combined with \eqref{ineqwulff} and
$\Delta u \equiv P_H(\Omega)/ \vert \Omega \vert$,
gives
\begin{equation}
\vert W \vert \le \left( \frac{P_H(\Omega)}
{n \vert \Omega \vert} \right)^n \vert \Gamma_u \vert
\le \left( \frac{P_H(\Omega)}
{n \vert \Omega \vert} \right)^n \vert \Omega \vert  .
\label{contact}
\end{equation}
Finally, using that $P_H(W) = n \vert W\vert$ ---see \eqref{per/vol}---, we
conclude that
\begin{equation}
\label{isopfin}
\frac{P_H(W)}{\vert W \vert^{\frac{n-1}{n}}}=
n\vert W \vert^{\frac{1}{n}} \le
\frac{P_H(\Omega)}{\vert \Omega \vert^{\frac{n-1}{n}}}.
\end{equation}

Note that when $\Omega=W$ then the solution of \eqref{eqsem} is $u(x)=\vert x\vert^2/2$ since $\Delta u=n$ and $u_\nu(x)=x\cdot\nu(x)=H\bigl(\nu(x)\bigr)$ a.e. on $\partial W$ ---recall \eqref{normalW}.
In particular, $\nabla u = {\rm Id}$ and all the eigenvalues of $D^2u(x)$ are equal.
Therefore, it is clear that all inequalities (and inclusions) in \eqref{gradmap}-\eqref{isopfin} are
equalities when $\Omega=W$.
This explains why the proof gives the best constant in the inequality.

Let us see next that, when $H|_{S^{n-1}}$ is smooth, the Wulff shaped domains
$\Omega=aW +b$ are the only smooth
domains for which equality occurs in \eqref{ppp}.
Indeed,
if \eqref{isopfin} is an equality then all the inequalities in \eqref{ineqwulff},
\eqref{means}, and \eqref{contact} are also equalities. In particular, we have
$\vert \Gamma_u \vert = \vert \Omega\vert$. Since $\Gamma_u \subset \Omega$,
$\Omega$ is an open set, and $\Gamma_u$ is closed relatively to $\Omega$,
we deduce that $\Gamma_u = \Omega$.

Recall that the geometric and arithmetic means of $n$ nonnegative numbers
are equal if and only if these $n$ numbers are all equal. Hence, the equality
in \eqref{means}
and the fact that $\Delta u$ is constant in $\Omega$ give that
$D^2 u = a{\rm Id}$ in all $\Gamma_u = \Omega$, where $ {\rm Id}$
is the identity matrix and
$a=P_H(\partial\Omega) /(n\vert\Omega\vert)$
is a positive constant. Let $x_0 \in \Omega$ be any given point.
Integrating $D^2u=a{\rm Id}$ on segments from $x_0$,
we deduce that
$$
u(x)=u(x_0)+\nabla u(x_0) \cdot (x-x_0) + \frac{a}{2}\, \vert x-x_0 \vert^2
$$
for $x$ in a neighborhood of $x_0$.
In particular,
$\nabla u (x) = \nabla u (x_0) + a(x-x_0)$ in such a
neighborhood, and hence the map
$\nabla u  - a\iden$ is locally constant.
Since $\Omega$ is connected, we deduce that the map $\nabla u  - a\iden$ is indeed
a constant, say $\nabla u  - a\iden\equiv y_0$.

It follows that
$\nabla u (\Gamma_u) = \nabla u (\Omega) = y_0 + a\Omega$.
By \eqref{gradmap} we know that $W\subset \nabla u (\Gamma_u)=
y_0 + a\Omega$.
In addition,
these two sets have the same measure since equalities occur in \eqref{ineqwulff}.
Thus, $y_0+a\Omega$ is equal to $W$ up to a set of measure zero.
In fact, in the present situation, since  $W$ is convex and $y_0+a\Omega$ is open,
one easily proves that $W =
y_0 + a\Omega$. Hence, $\Omega$ is of the form $\tilde{a} W+\tilde{b}$ for some $\tilde{a}>0$ and
$\tilde{b}\in \R^n$.
\end{proof}

\bigskip
\section{Proof of Theorem \ref{th1}: the case $w\equiv0$ on $\partial\Sigma$ and $H=\|\cdot\|_2$}
\label{sec3}

For the sake of clarity, we present in this section the proof of Theorem \ref{th1} under the assumptions $w\equiv0$ on $\partial\Sigma$ and $H=\|\cdot\|_2$.
The proof is simpler in this case.
Within the proof we will use the following lemma.

\begin{lem}\label{tic}
Let $w$ be a positive homogeneous function of degree $\alpha>0$ in an open cone $\Sigma\subset\R^n$. Then, the following conditions are equivalent:
\begin{itemize}
\item For each $x,z\in\Sigma$, it holds the following inequality:
\[\alpha\left(\frac{w(z)}{w(x)}\right)^{1/\alpha}\leq \frac{\nabla w(x)\cdot z}{w(x)}.\]
\item The function $w^{1/\alpha}$ is concave in $\Sigma$.
\end{itemize}
\end{lem}

\begin{proof} Assume first $\alpha=1$.
A function $w$ is concave in $\Sigma$ if and only if for each $x,z\in\Sigma$ it holds
\begin{equation}\label{lemap1}w(x)+\nabla w(x)\cdot (z-x)\geq w(z).\end{equation}
Now, since $w$ is homogeneous of degree 1, we have
\begin{equation}\label{lemap2}\nabla w(x)\cdot x=w(x).\end{equation}
This can be seen by differentiating the equality $w(tx)=tw(x)$ and evaluating at $t=1$.
Hence, from \eqref{lemap1} and \eqref{lemap2} we deduce that an homogeneous function $w$ of degree 1 is concave if and only if
\[w(z)\leq \nabla w(x)\cdot z.\]
This proves the lemma for $\alpha=1$.

Assume now $\alpha\neq1$.
Define $v=w^{1/\alpha}$, and apply the result proved above to the function $v$, which is homogeneous of degree 1.
We obtain that $v$ is concave if and only if
\[v(z)\leq {\nabla v(x)\cdot z}.\]
Therefore, since $\nabla v(x)=\alpha^{-1}w(x)^{\frac{1}{\alpha}-1}\nabla w(x)$, we deduce that
$w^{1/\alpha}$ is concave if and only if
\[w(z)^{1/\alpha}\leq \frac{\nabla w(x)\cdot z}{\alpha w(x)^{1-\frac{1}{\alpha}}},\]
and the lemma follows.
\end{proof}

We give now the

\begin{proof}[Proof of Theorem \ref{th1} in the case $w\equiv0$ on $\partial\Sigma$ and $H=\|\cdot\|_2$]
For the sake of simplicity we assume here that $E= U\cap \Sigma$, where $U$ is some bounded
smooth domain in $\R^n$.
The case of general sets will be treated in Section \ref{sec4} when we prove Theorem \ref{th1} in its full generality.

Observe that since $E= U\cap \Sigma$ is piecewise Lipschitz, and $w\equiv0$ on $\partial\Sigma$, it holds
\begin{equation}
\label{perimeteronlyw}
P_{w}(E ; \Sigma) = \int_{\partial U\cap \Sigma} w(x) dx = \int_{\partial E} w(x)dx.
\end{equation}
Hence, using that $w\in C(\overline \Sigma)$ and \eqref{perimeteronlyw}, it is immediate to prove that for any $y\in \Sigma$ we have
\[  \lim_{\delta \downarrow 0} P_w(E + \delta y;\Sigma) = P_w(E;\Sigma) \quad \mbox{and} \quad \lim_{\delta \downarrow 0} w(E + \delta y) = w(E).\]
We have denoted $E + \delta y  = \{x + \delta y \, , \ x\in E\}$.
Note that $P_w(E + \delta y;\Sigma)$ could not converge to $P_w(E;\Sigma)$ as $\delta \downarrow 0$ if $w$ did not vanish on the boundary of the cone $\Sigma$.

By this approximation property and a subsequent regularization of $E + \delta y$ (a detailed argument can be found in the proof of Theorem \ref{th1} in next section), we see
that it  suffices to prove \eqref{mainresult} for smooth domains whose closure is contained in $\Sigma$.
Thus, from now on in the proof we denote by $\Omega$, instead of $E$, any smooth domain satisfying $\overline\Omega\subset \Sigma$.
We next show \eqref{mainresult} with $E$ replaced by $\Omega$.

At this stage, it is clear that by approximating $w|_{\overline{\Omega}}$ we can assume $w\in C^\infty(\overline\Omega)$.

Let $u$ be a solution of the linear Neumann problem
\begin{equation}
\left\{
\alignedat2
w^{-1}\textrm{div}(w\nabla u) &= b_\Omega
&\quad &\text{in } \Omega\quad \textrm{(with }\overline\Omega\subset\Sigma\textrm{)}\\
\frac{\partial u}{\partial\nu} &=1 &\quad &\text{on }\partial \Omega .
\endalignedat
\right.
\label{eqsem2}
\end{equation}
The Fredholm alternative ensures that there exists a solution of \eqref{eqsem2} (which is unique up to an additive constant) if and only if the constant $b_\Omega$ is given by
\begin{equation}\label{cttb}
b_\Omega=\frac{P_w(\Omega;\Sigma)}{w(\Omega)}.\end{equation}
Note also that since $w$ is positive and smooth in $\overline\Omega$, \eqref{eqsem2} is a uniformly elliptic problem with smooth coefficients.
Thus, $u\in C^{\infty}(\overline\Omega)$.
For these classical facts, see Example 2 in Section 10.5 of \cite{H}, or the end of Section 6.7 of \cite{GT}.

Consider now the lower contact set of $u$, $\Gamma_u$, defined by \eqref{lcset} as the set of points in $\Omega$ at which the
tangent hyperplane to the graph of $u$ lies below $u$ in all $\overline \Omega$.
Then, as in the proof of the Wulff theorem in Section \ref{sec2}, we touch by below the graph of $u$ with hyperplanes of fixed slope $p\in B_1$, and using the boundary condition in \eqref{eqsem2} we deduce that $B_1 \subset \nabla u (\Gamma_u)$.
From this, we obtain
\begin{equation}\label{fivepointstar}
B_1\cap \Sigma\subset \nabla u(\Gamma_u)\cap \Sigma
\end{equation}
and thus
\begin{equation}\label{ineqsec3}
\begin{split}
w(B_1\cap \Sigma) &\leq \int_{\nabla u (\Gamma_u)\cap\Sigma}w(p)dp \\
&\leq \int_{\Gamma_u\cap (\nabla u)^{-1}(\Sigma)} w(\nabla u(x))\det D^2u(x)\,dx\\
&\leq  \int_{\Gamma_u\cap (\nabla u)^{-1}(\Sigma)} w(\nabla u)\left(\frac{\Delta u}{n}\right)^ndx.
\end{split}
\end{equation}
We have applied the area formula to the smooth map $\nabla u : \Gamma_u \rightarrow \mathbb R^n$ and also the classical arithmetic-geometric means inequality ---all eigenvalues of $D^2u$ are nonnegative in $\Gamma_u$ by definition of this set.

Next we use that, when $\alpha>0$,
\[s^{\alpha}t^n\leq \left(\frac{\alpha s+nt}{\alpha+n}\right)^{\alpha+n}\ \ \textrm{for all }\ s>0\ \textrm{and}\ t>0,\]
which follows from the concavity of the logarithm function. Using also Lemma \ref{tic}, we find
\[\frac{w(\nabla u)}{w(x)}\left(\frac{\Delta u}{n}\right)^n\leq
\left(\frac{\alpha\left(\frac{w(\nabla u)}{w(x)}\right)^{1/\alpha}+\Delta u}{\alpha+n}\right)^{\alpha+n}\leq
\left(\frac{\frac{\nabla w(x)\cdot \nabla u}{w(x)}+\Delta u}{D}\right)^{D}.\]
Recall that $D=n+\alpha$.
Thus, using the equation in \eqref{eqsem2}, we obtain
\begin{equation}\label{36}
\frac{w(\nabla u)}{w(x)}\left(\frac{\Delta u}{n}\right)^n\leq \left(\frac{b_\Omega}{D}\right)^{D}\ \ {\rm in}\ \Gamma_u\cap (\nabla u)^{-1}(\Sigma).
\end{equation}
If $\alpha=0$ then $w\equiv1$, and \eqref{36} is trivial.

Therefore, since $\Gamma_u\subset \Omega$, combining \eqref{ineqsec3} and \eqref{36} we obtain
\begin{equation}\label{7}\begin{split}
w(B_1\cap\Sigma) &\leq \int_{\Gamma_u\cap (\nabla u)^{-1}(\Sigma)} \left(\frac{b_\Omega}{D}\right)^{D}w(x)dx=
\left(\frac{b_\Omega}{D}\right)^{D}w(\Gamma_u\cap (\nabla u)^{-1}(\Sigma))\\
&\leq \left(\frac{b_\Omega}{D}\right)^{D}w(\Omega)
= D^{-D}\frac{P_w(\Omega;\Sigma)^{D}}{w(\Omega)^{D-1}}.\end{split}\end{equation}
In the last equality we have used the value of the constant $b_\Omega$, given by \eqref{cttb}.

Finally, using that, by \eqref{formula-per=Dvol-for-Wulff}, we have $P_{w}(B_1;\Sigma) = D\,w(B_1\cap \Sigma)$, we obtain the desired inequality \eqref{mainresult}.

An alternative way to see that \eqref{7} is equivalent to \eqref{mainresult} is to analyze the previous argument when $\Omega=B_1\cap \Sigma$.
In this case $\overline\Omega\nsubseteq\Sigma$ and therefore, as explained in Section \ref{elementsproof}, we must solve problem \eqref{pb_neu_pes} instead of problem \eqref{eqsem2}.
When $\Omega=B_1\cap \Sigma$ the solution to problem \eqref{pb_neu_pes} is $u(x)=|x|^2/2$.
For this function $u$ we have $\Gamma_u=B_1 \cap \Sigma$ and $b_{B_1\cap \Sigma} = P_w(B_1;\Sigma)/w(B_1\cap \Sigma)$ ---as in  \eqref{cttb}.
Hence, for these concrete $\Omega$ and $u$ one verifies that all inclusions and inequalities in \eqref{fivepointstar}, \eqref{ineqsec3}, \eqref{36}, \eqref{7} are equalities, and thus \eqref{mainresult} follows.
\end{proof}

\bigskip
\section{Proof of Theorem \ref{th1}: the general case}
\label{sec4}

In this section we prove Theorem \ref{th1} in its full generality.
At the end of the section, we include the geometric argument of E. Milman that
provides an alternative proof of Theorem \ref{th1} in the case that the exponent $\alpha$ is an integer.

\begin{proof}[Proof of Theorem \ref{th1}]
Let
\[W_0:=W\cap \Sigma,\]
an open convex set, and nonempty by assumption.
Since $\lambda W_0\subset W_0$ for all $\lambda\in (0,1)$, we deduce that $0\in \overline W_0$.
Therefore, as commented in subsection \ref{subs1.1}, there is a unique gauge $H_0$ such that its Wulff shape is $W_0$.
In fact, $H_0$ is defined by expression \eqref{uniqueH} (with $W$ and $H$ replaced by $W_0$ and $H_0$).

Since $H_0\le H$ we have
\[P_{w,H_0}(E;\Sigma)\leq P_{w,H}(E;\Sigma)\ \ {\rm for\  each\  measurable\  set}\ E,\]
while, using \eqref{defperiint},
\[P_{w,H_0}(W_0;\Sigma)=P_{w,H}(W;\Sigma)\quad \textrm{and}\quad w(W_0)=w(W\cap\Sigma).\]
Thus, it suffices to prove that
\begin{equation}\label{ineqH0}
\frac{P_{w,H_0}(E;\Sigma)}{w(E)^{\frac{D-1}{D}} }\geq  \frac{P_{w,H_0}(W_0;\Sigma)}{w(W_0)^{\frac{D-1}{D}}}
\end{equation}
for all measurable sets $E\subset\Sigma$ with $w(E)<\infty$.

The definition of $H_0$ is motivated by the following reason. Note that $H_0$ vanishes on the directions
normal to the cone $\Sigma$. Thus, by considering $H_0$ instead of $H$, we will be able (by an approximation argument) to assume that $E$ is a smooth domain whose closure is contained in $\Sigma$.
This approximation cannot be done when $H$ does not vanish on the directions normal to the cone ---since
the relative perimeter does not count the part of the boundary lying on $\partial\Sigma$, while when $\overline E\subset\Sigma$ the whole perimeter is counted.

We split the proof of \eqref{ineqH0} in three cases.

{\em Case 1.} Assume that $E=\Omega$, where $\Omega$ is a smooth domain satisfying $\overline\Omega\subset\Sigma$.

At this stage, it is clear that by regularizing $w|_{\overline{\Omega}}$ and $H_0|_{S^{n-1}}$ we can assume $w\in C^\infty(\overline\Omega)$ and $H_0 \in C^\infty(S^{n-1})$.

Let $u$ be a solution to the Neumann problem
\begin{equation}
\left\{ \alignedat2 w^{-1}\textrm{div}(w\nabla u) &= b_\Omega
&\quad &\text{in } \Omega\\
\frac{\partial u}{\partial\nu} &=H_0(\nu) &\quad &\text{on }\partial \Omega ,
\endalignedat
\right. \label{eqsem3}
\end{equation}
where $b_\Omega\in\R$ is chosen so that the problem has a unique solution up to an additive constant, that is,
\begin{equation}\label{cttbsec4}
b_\Omega=\frac{P_{w,H_0}(\Omega;\Sigma)}{w(\Omega)}.
\end{equation}
Since $w$ is positive and smooth in $\overline\Omega$, 
and $H_0$, $\nu$, and $\Omega$ are smooth, we have that $u\in C^{\infty}(\overline\Omega)$.
See our comments following \eqref{eqsem2}-\eqref{cttb} for references of these classical facts.

Consider the lower contact set of $u$, defined by
\[\Gamma_u =\{ x \in \Omega \ : \ u(y) \ge u(x) + \nabla u (x) \cdot (y-x)\  \text{ for all } y \in \overline \Omega \}.\]
We claim that
\begin{equation} \label{claimW}
W_0 \subset \nabla u (\Gamma_u)\cap \Sigma.
\end{equation}

To prove \eqref{claimW}, we proceed as in the proof of Theorem \ref{wulffthm} in Section \ref{sec2}.
Take $p\in W_0$, that is, $p\in \mathbb R^n$ satisfying $p\cdot\nu < H_0(\nu)$ for each
$\nu\in S^{n-1}$.
Let $x\in \overline \Omega$ be a point such that
\[\min_{y\in \overline \Omega} \,\{ u(y) -p\cdot y \} = u(x)-p\cdot x.\]
If $x\in \partial \Omega$ then the exterior normal derivative of $u(y)-p\cdot y$ at $x$
would be nonpositive and, hence, $(\partial u /\partial\nu) (x) \le p\cdot \nu < H_0(p)$, a contradiction with \eqref{eqsem3}.
Thus, $x\in\Omega$, $p=\nabla u(x)$, and $x\in \Gamma_u$ ---see Section \ref{sec2} for more details. Hence, $W_0\subset \nabla u(\Gamma_u)$, and since $W_0\subset \Sigma$, claim \eqref{claimW} follows.

Therefore,
\begin{equation}
w(W_0) \leq \int_{\nabla u (\Gamma_u)\cap \Sigma}w(p) dp \leq \int_{\Gamma_u\cap (\nabla u)^{-1}(\Sigma)}w(\nabla u) \det D^2u \ dx . \label{ineq}
\end{equation}
We have applied the area formula to the smooth map $\nabla u : \Gamma_u \rightarrow \mathbb R^n$, and we have used that its Jacobian, $\det D^2u$, is
nonnegative in $\Gamma_u$ by definition of this set.

We proceed now as in Section \ref{sec3}.
Namely, we first use the following weighted version of the inequality between the arithmetic and the geometric means,
\begin{equation*}\label{am-gm}
a_0^{\alpha}a_1\cdots a_n\leq \left(\frac{\alpha a_0+a_1+\cdots+a_n}{\alpha+n}\right)^{\alpha+n},
\end{equation*}
applied to the numbers
$a_0=\left(\frac{w(\nabla u)}{w(x)}\right)^{1/\alpha}$ and $a_i=\lambda_i(x)$ for $i=1,...,n$, where $\lambda_1,...,\lambda_n$ are the eigenvalues of $D^2u$.
We obtain
\begin{equation}\label{y}
\frac{w(\nabla u)}{w(x)}\det D^2u\leq
\left(\frac{\alpha\left(\frac{w(\nabla u)}{w(x)}\right)^{1/\alpha}+\Delta u}{\alpha+n}\right)^{\alpha+n}\leq
\left(\frac{\frac{\nabla w(x)\cdot \nabla u}{w(x)}+\Delta u}{\alpha+n}\right)^{\alpha+n}.
\end{equation}
In the last inequality we have used Lemma \ref{tic}.
Now, the equation in \eqref{eqsem3} gives
\[\frac{\nabla w(x)\cdot \nabla u}{w(x)}+\Delta u=\frac{\div(w(x)\nabla u)}{w(x)} \equiv b_\Omega,\]
and thus using \eqref{cttbsec4} we find
\begin{equation}\label{ultima}
\begin{split}
\int_{\Gamma_u\cap (\nabla u)^{-1}(\Sigma)}w(\nabla u) \det D^2u  \ dx &
\leq \int_{\Gamma_u\cap (\nabla u)^{-1}(\Sigma)}w(x)\left(\frac{b_\Omega}{D}\right)^Ddx \\
&\leq \int_{\Gamma_u}w(x)\left(\frac{b_\Omega}{D}\right)^Ddx =\left( \frac{P_{w,H_0}(\Omega;\Sigma)} {D\,w(\Omega)} \right)^Dw(\Gamma_u).
\end{split}
\end{equation}

Therefore, from \eqref{ineq} and \eqref{ultima} we deduce
\begin{equation}\label{hola}
w(W_0)\leq \left( \frac{P_{w,H_0}(\Omega;\Sigma)} {D\,w(\Omega)} \right)^D w(\Gamma_u)
\leq \left( \frac{P_{w,H_0}(\Omega;\Sigma)} {D\,w( \Omega )} \right)^D w(\Omega).
\end{equation}

Finally, using that, by \eqref{formula-per=Dvol-for-Wulff}, we have $P_{w,H_0}(W;\Sigma) = D\,w(W_0)$, we deduce \eqref{ineqH0}.

An alternative way to see that \eqref{hola} is equivalent to \eqref{ineqH0} is to analyze the previous argument when $\Omega=W_0=W\cap \Sigma$.
In this case $\overline\Omega \nsubseteq\Sigma$ and therefore, as explained in Section \ref{elementsproof}, we must solve problem
\begin{equation} \label{finalpb}
\left\{
\alignedat2
w^{-1}{\rm div\,} (w\nabla u) &= b_\Omega&
\quad &\text{in } \Omega
\\ \vspace{1mm}
\frac{\partial u}{\partial\nu} &=H_0(\nu)&\quad &\text{on }\partial \Omega \cap \Sigma
\\ \vspace{1mm}
\frac{\partial u}{\partial\nu} &= 0&\quad &\text{on }\partial \Omega \cap \partial\Sigma
\endalignedat\right.
\end{equation}
instead of problem \eqref{eqsem3}.
When $\Omega=W_0$, the solution to problem \eqref{finalpb} is
\[u(x)=|x|^2/2.\]
For this function $u$ we have $\Gamma_u=W_0$ and $b_{W_0} = P_{w,H_0}(W_0;\Sigma)/w(W_0)$ ---as in \eqref{cttbsec4}.
Hence, for these concrete $\Omega$ and $u$ one verifies that all inclusions and inequalities in \eqref{claimW}, \eqref{ineq}, \eqref{y}, \eqref{ultima}, and \eqref{hola} are equalities, and thus \eqref{ineqH0} follows.

{\em Case 2.} Assume now that  $E= U\cap \Sigma$, where $U$ is a bounded smooth open set in $\R^n$. 
Even that both $U$ and $\Sigma$ are Lipschitz sets, their intersection might not be Lipschitz (for instance if $\partial U$ and $\partial\Sigma$ meet tangentially at a point). As a consequence, approximating $U\cap\Sigma$ by smooth sets converging in perimeter is a more subtle issue. However,
we claim that there exists a sequence $\{\Omega_k\}_{k\ge 1}$ of smooth bounded domains satisfying 
\begin{equation}\label{regularization}
 \overline{\Omega_k}\subset \Sigma\qquad \mbox{and} \qquad\lim_{k\rightarrow\infty} \frac{P_{w,H_0}(\Omega_k;\Sigma)}{w(\Omega_k)^{\frac{D-1}{D}} }\le \frac{P_{w,H_0}(E;\Sigma)}{w(E)^{\frac{D-1}{D}}}.
\end{equation}
Case 2 follows immediately using this claim and what we have proved in Case 1. We now proceed to prove the claim.

It is no restriction to assume that $e_n$, the $n$-th vector of the standard basis, 
belongs to the cone $\Sigma$. Then, $\partial \Sigma$ is a convex graph (and therefore, 
Lipschitz in every compact set) over the variables $x_1,\dots,x_{n-1}$. That is,
\begin{equation}\label{Sigmasuperlevel}
\Sigma=\{x_n > g(x_1,\dots, x_{n-1})\}
\end{equation}
for some convex function $g:\R^{n-1}\rightarrow \R$.

First we construct a sequence
\begin{equation}\label{Fksuperlevel}
F_k =\{x_n > g_k(x_1,\dots, x_{n-1})\}, \quad {k\ge1}
\end{equation}
of convex smooth sets whose boundary is a graph $g_k:\R^{n-1}\rightarrow \R$ over the first 
$n-1$ variables and satisfying:
\begin{enumerate}
\item[(i)]  $g_1>g_2> g_3 >\dots\ $ in $\overline B$, where $B$ is a large ball $B\subset \R^{n-1}$ containing 
the projection of $\overline U$.
\item[(ii)] $g_{k} \rightarrow g$  uniformly in $\overline B$.
\item[(iii)] $\nabla g_k \rightarrow  \nabla g$ almost everywhere in $\overline B$ and 
$|\nabla g_k|$ is bounded independently of~$k$.
\item[(iv)] The smooth manifolds $\partial F_k= \{x_n =  g_k(x_1,\dots, x_{n-1})\}$ and $\partial U$ 
intersect transversally.
\end{enumerate}
To construct the sequence $g_k$, we consider the convolution of $g$ with a standard mollifier
\[\tilde g_k = g \ast k^{n-1}\eta(k x)+ \frac{C}{k}\]
with $C$ is a large constant (depending on $\|\nabla g\|_{L^\infty(\R^{n-1})}$) to guarantee $\tilde g_k>g$ in $\overline B$. It follows that a subsequence of $\tilde g_k$ will satisfy (i)-(iii).
Next, by a version of Sard's Theorem \cite[Section 2.3]{topology} almost every small translation of the 
smooth manifold $\{x_n =  \tilde g_k(x_1,\dots, x_{n-1})\}$ will intersect $\partial U$ transversally. 
Thus, the sequence
\[g_k(x_1,\dots, x_{n-1}) = \tilde g_k(x_1- y_1^k, \dots, x_{n-1}-y_{n-1}^k) + y_n^k\]
will satisfy (i)-(iv) if $y^k\in \R^{n}$ are chosen with $|y^k|$ sufficiently small depending on~$k$
---in particular to preserve (i).

Let us show now that $P_{w,H_0}(U \cap F_k ; \Sigma)$ converges to $P_{w,H_0}(E ; \Sigma)$ as 
$k\uparrow\infty$. Note that (i) yields $F_{k}\subset F_{k+1}$ for all $k\ge 1$. 
This monotonicity will be useful to prove the convergence of perimeters, that we do next.

Indeed, since we considered the gauge $H_0$ instead of $H$, we have the following property
\begin{equation}
\label{perimeter}
P_{w,H_0}(E ; \Sigma) = \int_{\partial U\cap \Sigma} H_0(\nu(x))w(x) dx = \int_{\partial E} H_0(\nu(x))w(x)dx.
\end{equation}
This is because 
$\partial E = \partial(U\cap E)\subset (\partial U \cap\Sigma)\cup (\overline U \cap \partial \Sigma)$ and
\begin{equation}
\label{H0}
H_0(\nu(x))=0\quad \mbox{ for almost all }\ x \in \partial\Sigma.
\end{equation}
Now, since $\partial (U\cap F_k) \subset  (\partial U\cap F_k) \cup (\overline U\cap\partial F_k)$ we have
\[0 \le P_{w,H_0}(U \cap F_k ; \Sigma) -  \int_{\partial U\cap F_k} H_0(\nu(x))w(x) dx \le  \int_{\overline U\cap\partial F_k} H_0(\nu_{F_k}(x))w(x) dx.
\]
On one hand, using dominated convergence, \eqref{Sigmasuperlevel}, \eqref{Fksuperlevel}, (ii)-(iii), and \eqref{H0},  we readily prove that
\[\int_{\overline U\cap\partial F_k} H_0(\nu_{F_k}(x))w(x) dx \rightarrow 0.\]
On the other hand,  by (i) and (ii), $F_k\cap (B\times\R)$ is an increasing sequence exhausting  $\Sigma\cap (B\times\R)$. Hence, by monotone convergence
\[\int_{\partial U\cap F_k} H_0(\nu(x))w(x) dx \rightarrow \int_{\partial U\cap \Sigma} H_0(\nu(x))w(x) dx = P_{w,H_0}(E ; \Sigma).\]
Therefore, the sets $U\cap F_k$ approximate $U\cap \Sigma$ in $L^1$ and in the $(w,\negthinspace H_0)$-perimeter. Moreover, by (iv), $U\cap F_k$ are Lipschitz open sets.

Finally, to obtain the sequence of smooth domains $\Omega_k$ in \eqref{regularization}, 
we use a partition of unity and local regularization of the Lipschitz sets $U\cap F_k$ 
to guarantee the convergence of the $(w,\negthinspace H_0)$-perimeters.
In case that the regularized sets had more than one connected component, 
we may always choose the one having better isoperimetric quotient.

{\em Case 3.}
Assume that $E$ is any measurable set with $w(E)<\infty$ and $P_{w,H_0}(E;\Sigma)\leq P_{w,H}(E;\Sigma)< \infty$.
As a consequence of Theorem 5.1 in \cite{BBF}, $C^\infty_c(\R^n)$ is dense in the space $BV_{\mu,H_0}$ of functions of bounded variation with respect to the measure $\mu=w\chi_{\Sigma}$ and the gauge $H_0$.
Note that our definition of perimeter $P_{w,H_0}( E;\Sigma)$ coincides with the $(\mu,H_0)$-total variation of the characteristic function $\chi_E$, that is, $|D_{\hspace{-0.7pt}\mu}\hspace{1pt}\chi_E |_{H_0}$ in notation of  \cite{BBF}.
Hence, by the coarea formula in Theorem 4.1 in \cite{BBF} and the argument in Section 6.1.3 in \cite{Mazja}, we find that for each measurable set $E\subset\Sigma$ with finite measure there exists a sequence of bounded smooth sets $\{U_k\}$ satisfying
\[\lim_{k\rightarrow\infty} w(U_k\cap \Sigma)= w(E)\quad  {\rm and}\quad \lim_{k\rightarrow\infty} P_{w,H_0}(U_k; \Sigma)= P_{w,H_0}(E;\Sigma).\]
Then we are back to Case 2 above, and hence the proof is finished.
\end{proof}

After the announcement of our result and proof in \cite{CRS-CRAS},
Emanuel Milman showed us a nice geometric construction that yields the weighted inequality in Theorem \ref{th1} in the case that $\alpha$ is a nonnegative \emph{integer}.
We next sketch this construction.

\begin{rem}[Emanuel Milman's construction]\label{milmans}
When $\alpha$ is a nonnegative integer the weighted isoperimetric inequality of Theorem \ref{th1} (when $H=\|\cdot\|_2$) can be proved as a limit case of the Lions-Pacella inequality in convex cones of $\R^{n+\alpha}$.
Indeed, let $w^{1/\alpha}>0$ be a concave function, homogeneous of degree 1, in an open convex cone $\Sigma\subset\R^n$.
For each $\varepsilon>0$, consider the cone
\[
\mathcal C_{\varepsilon} = \left\{ (x,y)\in \R^n\times \R^\alpha \,: \, x\in \Sigma,\ |y|<\varepsilon w(x)^{1/\alpha}\right\}.
\]
From the convexity of $\Sigma$ and the concavity of $w^{1/\alpha}$ we have  that $\mathcal C_\varepsilon$ is a convex cone.
Hence, by Theorem \ref{isopcon} we have
\begin{equation}\label{desig-a-dalt}
\frac{P(\tilde E;\mathcal C_\varepsilon)}{|\tilde E\cap \mathcal C_\varepsilon|^{\frac{n+\alpha-1}{n+\alpha}}} \ge \frac{P(\mathcal B_1;\mathcal C_\varepsilon)}{|\mathcal B_1\cap \mathcal C_\varepsilon|^{\frac{n+\alpha-1}{n+\alpha}}}\quad \textrm{for all}\ \tilde E\ \textrm{with}\ |\tilde E\cap \mathcal C_\varepsilon|<\infty,
\end{equation}
where $\mathcal B_1$ is the unit ball of $\R^{n+\alpha}$.
Now, given a Lipschitz set $E\subset \R^n$, consider the cylinder $\tilde E= E\times \R^{\alpha}$
one finds
\[|\tilde E\cap \mathcal C_{\varepsilon}| = \int_{E\cap \Sigma} dx \int_{\{|y|<\varepsilon w(x)^{1/\alpha}\}} dy = \omega_\alpha\varepsilon^\alpha \int_{E\cap \Sigma} w(x) dx = \omega_\alpha\varepsilon^\alpha w(E\cap \Sigma)\]
and
\[P(\tilde E;\mathcal C_{\varepsilon}) = \int_{\partial E\cap \Sigma} dS(x) \int_{\{|y|<\varepsilon w(x)^{1/\alpha}\}} dy = \omega_\alpha \varepsilon^\alpha \int_{\partial E\cap \Sigma} w(x) dS = \omega_\alpha\varepsilon^\alpha P_w(E; \Sigma).\]

On the other hand, one easily sees that, as $\varepsilon\downarrow 0$,
\[\frac{P(\mathcal B_1;\mathcal C_\varepsilon)}{|\mathcal B_1\cap \mathcal C_\varepsilon|^{\frac{n+\alpha-1}{n+\alpha}}} =  (\omega_\alpha \varepsilon^\alpha)^{\frac{1}{n+\alpha}} \left(\frac{P_w(B_1; \Sigma)}{w(B_1\cap \Sigma)^{\frac{n+\alpha-1}{n+\alpha}}} + o(1)\right),\]
where $B_1$ is the unit ball of $\R^n$.
Hence, letting $\varepsilon\downarrow 0$ in \eqref{desig-a-dalt} one obtains
\[\frac{P_w(E; \Sigma)}{w(E\cap \Sigma)^{\frac{n+\alpha-1}{n+\alpha}}}\ge \frac{P_w(B_1; \Sigma)}{w(B_1\cap \Sigma)^{\frac{n+\alpha-1}{n+\alpha}}},\]
which is the inequality of Theorem \ref{th1} in the case that $H = \|\cdot\|_2$ and $\alpha$ is an integer.
\end{rem}

\end{document}